\documentclass[11pt]{article}
\usepackage{amsfonts}
\usepackage{amsmath}
\usepackage{amssymb}
\usepackage{amsthm}
\usepackage[mathscr]{euscript}
\usepackage{mathrsfs}
\usepackage{indentfirst}
\usepackage{enumitem}
\usepackage{color}
\usepackage{lipsum}
\numberwithin{equation}{section}

\newcommand{\lbl}[1]{\label{#1}}

\newtheorem{theo}{Theorem}[section]
\newtheorem{prop}{Proposition}[section]
\newtheorem{lem}{Lemma}[section]

\newtheorem{col}{Corollary}[section]
\newtheorem{defi}{Definition}[section]
\newcommand{\be}{\begin{equation}}
\newcommand{\ee}{\end{equation}}
\newcommand\bes{\begin{eqnarray}} \newcommand\ees{\end{eqnarray}}
\newcommand{\bess}{\begin{eqnarray*}}
\newcommand{\eess}{\end{eqnarray*}}

\newcommand\ep{\varepsilon}
\newcommand\kk{\left}
\newcommand\rr{\right}
\newcommand\dd{\displaystyle}

\newcommand\lm{\lambda}

\newcommand\yy{\infty}

\newcommand\R{\mathbb{R}}

\newcommand\ud{\underline}

\newcommand\qq{\eqref}\newcommand\dx{{\rm d}x}

\begin{document}

\begin{center}{\Large\bf Qualitative properties of solutions to nonlocal infectious SIR epidemic models}\footnote{This work was
supported by NSFC Grant 12171120.}\\[4mm]
 {\Large Hanxiang Bao,\; Mingxin Wang\footnote{Corresponding author. {\sl E-mail}: mxwang@hpu.edu.cn },\; Shaowen Yao}\\[0.5mm]
 {School of Mathematics and Information Science, Henan Polytechnic University, Jiaozuo  454003, China}
\end{center}

\begin{quote}
\noindent{\bf Abstract.} This paper studies qualitative properties of solutions of  nonlocal infectious SIR epidemic models \qq{1.3}-\qq{1.5}, with the homogeneous Neumann boundary conditions, Dirichlet boundary conditions and free boundary, respectively. We first use the upper and lower solutions method and the Lyapunov function method to prove the global asymptotically stabilities of the disease-free equilibrium and the unique positive equilibrium of \qq{1.3}. Then we use the theory of topological degree in cones to study the positive equilibrium solutions of \qq{1.4}, including the necessary and sufficient conditions for the existence, and the uniqueness in some special case. At last, for the free boundary problem \qq{1.5}, we study the longtime behaviors of solutions and criteria for spreading and vanishing. The highlights are to overcome failures of the  Lyapunov functional method and comparison principle, and difficulties in the maximum principle and Hopf boundary lemma of boundary value problems caused by nonlocal terms.


 \end{quote}

\pagestyle{myheadings}
\section{Introduction}

As seen in its long history, epidemic disease has caused orders of
magnitude more deaths around the world than warfare, floods or famines, so many efforts have been made to prevent and control their spread. In order to better understand the infectious mechanism of infectious diseases, many researchers are committed to establishing mathematical models of infectious diseases and studying the dynamical properties of infectious diseases.
The most influential theoretical SIR model was given by Kermack and McKendrick in 1927 (\cite{KM72}), when they studied the Great Plague of London, which took place in 1665-1666. The classical SIR Model is described by dividing populations into susceptible, infected, and recovered individuals, denoted by $S $, $I $ and $R $, respectively. The progress of individuals can be represented by
 \[S\to I\to R.\]

Assume that epidemic disease  are contact infections, and that $S,I$ and $R$ have the same fertility, and $R$ is immune and no longer infected. Li, Ni and Wang \cite{LNW2023} proposed the following SIR epidemic model with uniform spatial distribution
\bes\left\{\begin{array}{ll}
 \dd\frac{{\rm d}S}{{\rm d}t}=aN-\beta S-bNS-kIS,\;\;&t>0,\\[2mm]
 \dd\frac{{\rm d}I}{{\rm d}t}=k IS-\gamma I-\beta I-bNI,\;\;&t>0,\\ [2mm]
  \dd\frac{{\rm d}R}{{\rm d}t}=\gamma I-\beta R-bNR,\;\;&t>0,\\ [2mm]
  S(0)=S_0>0,\;\; I(0)=I_0>0,\;\;R(0)=0,
\end{array}\right.\label{1.1}
\ees
and the corresponding free boundary problem with local diffusion and double free boundaries
 \bes\left\{\begin{array}{ll}
 S_t-dS_{xx}=aN-\beta S-bNS-kIS,&t>0,\; x\in\mathbb{R},\\[1mm]
 I_t-dI_{xx}=kIS-\gamma I-\beta I-bNI,&t>0,\;x\in(g(t),h(t)),\\ [1mm]
 R_t-dR_{xx}=\gamma I-\beta R-bNR,&t>0,\;x\in(g(t),h(t)),\\ [1mm]
 g'(t)=-\mu I_x(t,g(t)),\ \ h'(t)=-\mu I_x(t,h(t)), &t>0,\\[1mm]
 I=R=0, &t\ge0, \; x\notin (g(t),h(t)),\\[1mm]
 S(0,x)=S_0(x),\ x\in\mathbb{R};\;\;\;\;
 I(0,x)=I_0(x),\  R(0,x)=0,&x\in[-h_0, h_0],\\[1mm]
 -g(0)=h(0)=h_0.
\end{array}\right.\label{1.2}
\ees
In the above two models, $N=S+I+R$ is the total population, $b=(a-\beta)/M$, $a$ and $\beta$ are the birth rate and intrinsic death rate, respectively, $M$ is the environmental carrying capacity, the term $bN$ in each equation represents depletion of the native resources by all populations \cite{MSB1986}; the term $kIS$ indicates that the disease is transmitted through contact and $k$ is the infection rate, while $\gamma$ is the recovery rate of the infected population; $R(0)=0$ means that there is no recovered individuals at the initial time. In the second model, it is assumed that $I$ and $R$ spread along the same free boundary since $R$ depends only on $I$, while $S$ distributes in the whole line $\R$.

For the model \qq{1.1}, Li, Ni and Wang derived its basic reproduction number $\mathcal{R}_0=\frac{k(a-\beta)}{b(a+\gamma)}$, and show that the disease-free equilibrium point is globally asymptotically stable if $\mathcal{R}_0\le1$, while the unique positive equilibrium point is globally asymptotically stable if $\mathcal{R}_0>1$. For the free boundary problem \qq{1.2}, they studied the well-posedness and longtime behaviors of solutions, and  obtained a spreading-vanishing dichotomy where the basic reproduction number $\mathcal{R}_0$ plays a crucial role.

Recently, Zhang and Wang \cite{ZhangWang} studied a version of \eqref{1.2} with nonlocal diffusion.

In the case of virus transmission through the air, susceptible populations at the point $x$ and time $t$ are not only affected (infected) by the infected individuals at that point, but also by those around them. Infection mechanism is subject to a nonlocal law
 \[\int_{\Omega} kP(x,y)I(y,t){\rm d}y=:k\mathcal{P}[I](x,t),\]
where $k>0$ is a constant, the kernel function $kP(x,y)$ represents the probability that the viruses carried by an infected individual at point $y$ spread to point $x$, which is nonnegative. In such case, the local infectious term $kIS$ in model \eqref{1.1} should be replaced by the nonlocal infectious term $k\mathcal{P}[I]S$, and the corresponding models should be the initial boundary value problems having the homogeneous Neumann boundary conditions
\bes\left\{\begin{array}{ll}
S_t-d\Delta S=aN-\beta S-bNS-k\mathcal{P}[I]S,\;\;&x\in\Omega,\;t>0,\\[1mm]
I_t-d\Delta I=k\mathcal{P}[I]S-\gamma I-\beta I-bNI,\;\;&x\in\Omega,\;t>0,\\[1mm]
R_t-d\Delta R=\gamma I-\beta R-bNR,\;\;&x\in\Omega,\;t>0,\\[2mm]
\dd\frac{\partial S}{\partial \nu}=\frac{\partial I}{\partial \nu}=\frac{\partial R}{\partial \nu}=0,\;\;&x\in\partial\Omega,\;t>0,\\[2mm]
	 S(x,0)=S_0(x)>0,\;\; I(x,0)=I_0(x)>0,\;\;R(x,0)=0, \;\;&x\in\overline{\Omega},
 \end{array}\right.\label{1.3}
 \ees
or the homogeneous Dirichlet boundary conditions
\bes\left\{\begin{array}{ll}
S_t-d\Delta S=aN-\beta S-bNS-k\mathcal{P}[I]S,\;\;&x\in\Omega,\;t>0,\\[1mm]
	 I_t-d\Delta I=k\mathcal{P}[I]S-\gamma I-\beta I-bNI,\;\;&x\in\Omega,\;t>0,\\[1mm]
	R_t-d\Delta R=\gamma I-\beta R-bNR,\;\;&x\in\Omega,\;t>0,\\[2mm]
    S=I=R=0,\;\;&x\in\partial\Omega,\;t>0,\\[2mm]
	 S(x,0)=S_0(x)>0,\;\; I(x,0)=I_0(x)>0,\;\;R(x,0)=0, \;\;&x\in\Omega,
\end{array}\right.\label{1.4}
 \ees
and the free boundary problem \eqref{1.2} should be the following model
 \bes\left\{\begin{array}{ll}
 S_t-d S_{xx}=aN-\beta S-bNS-k\mathcal{P}[I]S,\;\;&t>0,\; x\in\mathbb{R},\\[1mm]
 I_t-d I_{xx}=k\mathcal{P}[I]S-\gamma I-\beta I-bNI,\;\;&t>0,\;x\in(g(t),h(t)),\\ [1mm]
   R_t-d R_{xx}=\gamma I-\beta R-bNR,\;\;&t>0,\;x\in(g(t),h(t)),\\ [1mm]
 g'(t)=-\mu I_x(t,g(t)),\, h'(t)=-\mu I_x(t,h(t)),&t>0,\\[1mm]
   I=R=0, &t\ge0, \; x\notin (g(t),h(t)),\\[1mm]
   S=S_0(x)>0,\; x\in\mathbb{R};\;\;
   I=I_0(x)>0,\;R=R_0(x)=0,&t=0,\; x\in[-h_0, h_0],\\[1mm]
   -g(0)=h_0=h(0).
\end{array}\right.\label{1.5}
\ees
In \eqref{1.5},
 \[\mathcal{P}[I](t,x)=\int_\mathbb{R} P(x,y)I(t,y){\rm d}y=\int_{g(t)}^{h(t)} P(x,y)I(t,y){\rm d}y.\]

For problems \eqref{1.3} and \eqref{1.4}, we assume that $\partial \Omega\in C^{2+\alpha}$, and the following hold: \vspace{-2mm}
\begin{enumerate}
\item[{\bf(P1)}]\, $P\in C(\overline{\Omega}\times\overline{\Omega})$,\ $P(x,y)\ge0$ and $\int_{\Omega}P(x,y)\dx=1$ for any $y\in\Omega$ (normative of probability).\vspace{-2mm}
\item[{\bf(N)}]\, $(S_0, I_0)\in [C^2(\overline\Omega)]^2$, and the compatibility condition holds, i.e., $\frac{\partial S_0}{\partial \nu}(x)=\frac{\partial I_0}{\partial \nu}(x)=0$ in $\partial \Omega$ for the problem \eqref{1.3}.\vspace{-2mm}
\item[{\bf(D)}]\, $(S_0, I_0)\in [C^2(\overline\Omega)]^2$, and the compatibility condition holds, i.e., $S_0(x)=I_0(x)=0$ in $\partial\Omega$ for the problem \eqref{1.4}.\vspace{-2mm}
\end{enumerate}
For problems \eqref{1.5}, we assume that \vspace{-2mm}
\begin{enumerate}
\item[{\bf(P2)}]\, $P$ is bounded and locally Lipschitz continuous in $\mathbb{R}^2$,\ $P(x,y)=P(y,x)\ge0$ and $\int_{\mathbb{R}}P(x,y)d x=1$ for any $y\in\mathbb{R}$.\vspace{-2mm}
\item[{\bf(F)}]\, $S_0\in C^2(\mathbb{R})\cap L^\infty(\mathbb{R})$, $S_0>0$ in $\mathbb{R}$; $I_0\in C^2([-h_0,h_0])$, $I_0>0$ in $(-h_0,h_0)$, and
$I_0(\pm h_0)=0$;\vspace{-2mm}
\end{enumerate}

The main purpose of this paper is committed to research the dynamics of models \eqref{1.3}, \eqref{1.4} and \eqref{1.5}. In Section 2, we first prove the global existence and uniqueness of positive solution of \eqref{1.3}, and then investigate nonnegative equilibria and their stabilities. Specifically, we determine the basic reproduction number $\mathcal{R}_{01}$ and show that, by use of the upper and lower solutions method and the Lyapunov functional method, the disease-free equilibrium is globally asymptotically stable if $\mathcal{R}_{01}\le1$, while  the unique positive equilibrium is globally asymptotically stable if $\mathcal{R}_{01}>1$ and $a>\gamma$. In Section 3, we use the theory of topological degree in cones to study the positive equilibrium solutions of \eqref{1.4}, including the necessary and sufficient conditions for the existence, and the uniqueness in some special case. In the process, some properties of the principal eigenvalue of a nonlocal eigenvalue problem will be given. In Section 4, we mainly concern the well-posedness and longtime behaviors and criteria for spreading and vanishing of the problem \eqref{1.5}. The highlights of this paper are to overcome the inefficiency of the Lyapunov functional method and the failure of the comparison principle caused by nonlocal terms, and the difficulties brought by nonlocal terms in the maximum principle and Hopf boundary lemma of boundary value problems.

There have been many related works to study the nonlocal infectious epidemic models, the interested readers can refer to \cite{CLY17, HWzamp19, ChenW, YYW2023} and the references therein.

\section{The dynamical properties of the problem \eqref{1.3}}
\setcounter{equation}{0} {\setlength\arraycolsep{2pt}

In this section we first research the existence, uniqueness and positivity of the solution of \eqref{1.3}. Then we study the nonnegative constant equilibrium solutions and their stabilities. Throughout this section we always assume that hypotheses {\bf(P1)} and {\bf(N)} hold.

\subsection{Existence, uniqueness and positivity of the solution of\eqref{1.3}}

\begin{theo}\lbl{t2.1} The problem \qq{1.3} has a unique global solution $(S(x,t), I(x,t), R(x,t))$ and
 \bes 0< S, I, R\le
 \max\kk\{\dd\max_{\overline{\Omega}}\{S_0(x)+I_0(x)\},\,(a-\beta)/b\rr\}:=M
\;\;\;\text{for} \;x\in\overline{\Omega},\;t>0.
\nonumber\ees
Moreover, for any given $0<\alpha<1$ and $\tau>0$, $S,\,I,\,R\in C^{1+\alpha,\,(1+\alpha)/2}(\overline\Omega\times[\tau, \infty))$, and there exists a constant $C(\tau)$ such that
\begin{eqnarray}
\|S,\,I,\,R\|_{C^{1+\alpha,\,(1+\alpha)/2}(\overline\Omega\times[\tau, \infty))}\leq C(\tau).
\label{2.1}\end{eqnarray}
\end{theo}

\begin{proof} Set
 \bess\left\{\begin{array}{ll}
f_1(S,I,R)=a(S+I+R)-\beta S-b(S+I+R)S-k\mathcal{P}[I]S,\\[1mm]
f_2(S,I,R)=k\mathcal{P}[I]S-\gamma I-\beta I-b(S+I+R)I,\\[1mm]
f_3(S,I,R)=\gamma I-\beta R-b(S+I+R)R.
\end{array}\right.
 \eess
Let
\bess
A&=&\max_{\overline{\Omega}}\{S_0(x)+I_0(x)\}, \\[1mm]
B_i&=&\dd\max_{0\le{S,I,R}\le{A+1}}|f_i(S,I,R)|,\;\;i=1,2,3\eess
and
\bess
T_0=1/B.
\eess
For $(x,t)\in Q_{T_0}:=\Omega\times(0,{T}_0]$, we define
\bess\left\{\begin{array}{ll}
\bar{S}=\bar{I}=\bar{R}=A+t B,\\[1mm]\underline{S}=\underline{I}=\underline{R}=0.
\end{array}\right.\nonumber
\eess
It is easy to see that the following hold:
\bes\left\{\begin{array}{ll}
\bar S_t-d\Delta{\bar S}=B\ge f_1({\bar S},I,R), \;\;&\underline{I}\le I\le\bar I,\, \underline{R}\le R\le\bar R,\;(x,t)\in Q_{T_0},\\[2mm]
\underline{S}_t-d\Delta{\underline{S}}=0\le f_1(0,I,R),\;\;&\underline{I}\le I\le\bar I,\, \underline{R}\le R\le\bar R,\;(x,t)\in Q_{T_0},\\[2mm]
\bar{I}_t-d\Delta{\bar{I}}=B\ge f_2(S,{\bar{I}},R),\;\; &
\underline{S}\le S\le\bar S,\, \underline{R}\le R\le\bar R,\;(x,t)\in Q_{T_0},\\[2mm]
\underline{I}_t-d\Delta{\underline{I}}=0= f_2(S,0,R),\;\;&
\underline{S}\le S\le\bar S,\, \underline{R}\le R\le\bar R
,\;(x,t)\in Q_{T_0},\\[2mm]
\bar{R}_t-d\Delta{\bar{R}}=B\ge f_3(S,I,{\bar{R}}),\;\; &
\underline{S}\le S\le\bar S,\, \underline{I}\le I\le\bar I,\;(x,t)\in Q_{T_0},\\[2mm]
\underline{R}_t-d\Delta{\underline{R}}=0\le f_3(S,I,0),\;\;&
\underline{S}\le S\le\bar S,\, \underline{I}\le I\le\bar I,\;(x,t)\in Q_{T_0},\\[2mm]
    \dd\frac{\partial {{\bar{S}}}}{\partial \nu}=\frac{\partial {{\underline{S}}}}{\partial \nu}=\frac{\partial {{\bar{I}}}}{\partial \nu}=\frac{\partial {{\underline{I}}}}{\partial \nu}=\frac{\partial {{\bar{R}}}}{\partial \nu}=\frac{\partial {{\underline{R}}}}{\partial \nu}=0,\;&x\in\partial\Omega,\;t\in(0,{T}_0],\;
    \\[2mm]
   {\bar S}(x,0)\ge S(x,0)\ge{\underline S}(x,0),&x\in\Omega,\;
   \\[1mm]
    {\bar I}(x,0)\ge I(x,0)\ge{\underline I}(x,0),&x\in\Omega,\;
   \\[1mm]
    {\bar R}(x,0)\ge R(x,0)\ge{\underline R}(x,0),&x\in\Omega.
\end{array}\right.\nonumber
\ees
This shows that $({\bar S}, {\bar I}, {\bar R})$ and $({\underline S},{\underline I},{\underline R})$ is a pair of coupled upper and lower solutions of \qq{1.3} on $\Omega\times(0,{T}_0)$, see the \cite[Definition 4.3]{Wangpara}.

It is effortless to obtain that there exists a positive constant $M$ such that
 \bess
 |f_1(S_1 ,I_1, R_1)-f_1(S_2,I_2, R_2)|
 &\le& M\big(|S_1-S_2|+|I_1-I_2|+|R_1-R_2|+\mathcal{P}[|I_1-I_2|]\big),\;\;(x,t)\in Q_{T_0},\\
 |f_2(S_1 ,I_1, R_1)-f_2(S_2,I_2, R_2)|
 &\le& M\big(|S_1-S_2|+|I_1-I_2|+|R_1-R_2|+\mathcal{P}[|I_1-I_2|]\big),\;\;(x,t)\in Q_{T_0},\\
 |f_3(S_1 ,I_1, R_1)-f_3(S_2,I_2, R_2)|
 &\le& M\big(|S_1-S_2|+|I_1-I_2|+|R_1-R_2|\big),\;\;(x,t)\in Q_{T_0}
 \eess
for all
 \bess
 \underline{S}(x,t)\le S_i(x,t)\le\bar S(x,t),\,\underline{I}(x,t)\le I_i(x,t)\le\bar I(x,t),\, \underline{R}(x,t)\le R_i(x,t)\le\bar R(x,t),\;(x,t)\in Q_{T_0}.
 \eess
Take advantage of the arguments in the proof of \cite[Theorem 4.5]{Wangpara}, we can prove that the problem \qq{1.3} has a unique local solution $(S(x,t), I(x,t), R(x,t))$. Let $T_{\rm max}>0$ be the maximum existence time of
$(S(x,t), I(x,t), R(x,t))$.

Noticing that $N(x,t)=S(x,t)+I(x,t)+R(x,t)$ satisfies
\bes\left\{\begin{array}{ll}
N_t-d\Delta N=aN-\beta N-bN^2,\;\; &x\in\Omega,\;t\in(0,T_{\rm max}),\\[1mm]
\dd\frac{\partial N}{\partial \nu}
=0,\;\;&x\in\partial\Omega,\;t\in(0,T_{\rm max}),\\[2mm]
N_0(x)>0,
\;\;&x\in\Omega.
\end{array}\right.\nonumber
\ees
It is easy to show that, by the maximum principle,
 $$N(x,t)\le \max\kk\{\dd\max_{\overline{\Omega}}\{S_0(x)+I_0(x)\},\,(a-\beta)/b\rr\}=M,\;\;
 x\in\overline{\Omega},\;\,t\in(0,T_{\rm max}).$$
Hence, $S(x,t), I(x,t), R(x,t)\le M$ for $x\in\overline{\Omega}$ and $t\in(0,T_{\rm max})$. This implies that  $(S, I, R)$ exists globally in time $t$, i.e., $T_{\rm max}=+\infty$.

Recalling that $S(x,t), I(x,t), R(x,t)\ge 0$ and $S_0(x)>0$, $I_0(x)>0$, it is clear that $S(x,t), I(x,t)$, $R(x,t)>0$ for all $x\in\overline{\Omega}$ and $t>0$.

Furthermore, using \cite[Theorem 2.13]{Wangpara} we see that the estimate \eqref{2.1} holds.
\end{proof}

\subsection{Nonnegative constant equilibrium solutions and their stabilities for the model \qq{1.3}}

The possible nonnegative constant equilibrium solutions of \eqref{1.3} are the trivial one $E_0=(0,0,0)$, semi-trivial one
$E_1=((a-\beta)/b,0,0)$ and positive one $E_2=(S_*, I_*, R_*)$. Moreover, the positive constant equilibrium solution $E_2$ exists if and only if
 \bess
 {\cal R}_{01}=:\frac{k(a-\beta)}{b(a+\gamma)}>1, \eess
and is given uniquely by
 \bes
 S_*=\frac{a+\gamma}k,\,\;\;I_*=a\frac{k(a-\beta)-b(a+\gamma)}{bk(a+\gamma)},\,\;\;
 R_*=\gamma\frac{k(a-\beta)-b(a+\gamma)}{bk(a+\gamma)}
\nonumber\ees
when it exists. We define $N_*=S_*+I_*+R_*=(a-\beta)/b$.

In this subsection we shall research the stabilities of these three constant equilibrium solutions. Firstly, we easily obtain that $E_0$ is unstable. Then we want to prove that $E_1$ is globally asymptotically stable if $ {\cal R}_{01}\le1$.
Lastly, we intend to show that $E_2$ is globally asymptotically stable if $\mathcal{R}_{01}>1$ and $a>\gamma$.

Here we mention that the Lyapunov functional method is invalid for the problem \qq{1.3} since there are nonlocal terms (at least we don't know how to use it, right). On the other hand, because there is a nonlocal term $\mathcal{P}[I]$ in the first differential equation of \qq{1.3} and its coefficient $-kS$ is negative, the comparison principle may not hold true (won't prove it now). To overcome these difficulties, we find it useful to investigate the initial boundary value problem of $(S+I,\,I)$. We first prove that the upper and lower solutions method holds true for the problem of $(S+I,\,I)$, and then construct the appropriate upper and lower solutions by means of the solution of an initial value problem of an appropriate ordinary differential system. At last, we shall use the Lyapunov function method to show that the solution of the initial value problem of this ordinary differential system has the limit we need.

Let $V=S+I$, then $(V,I)$ satisfies
\bes\left\{\begin{array}{ll}
V_t-d\Delta V=aN(x,t)-\gamma I-[\beta+bN(x,t)]V,\;\;&x\in\Omega,\;t>0,\\[1mm]
I_t-d\Delta I=k\mathcal{P}[I](V-I)-\gamma I-\beta I-bN(x,t)I,\;\;&x\in\Omega,\;t>0,\\[1mm]
\dd\frac{\partial V}{\partial\nu}=\dd\frac{\partial I}{\partial\nu}
=0,\;\;&x\in\partial\Omega,\;t>0,\\[2mm]
V_0(x)>0,\;I_0(x)>0,\;\;&x\in\Omega.
\end{array}\right.\lbl{2.2}
\ees
Because $V-I=S>0$ in problem \eqref{2.2}, the comparison principle still holds although there is a nonlocal term $\mathcal{P}[I](x,t)=\int_{\Omega} P(x,y)I(t,y){\rm d}y$ in \eqref{2.2}. For the convenience of readers we shall prove the following comparison principle.

\begin{prop}{\rm(Comparison principle)}\lbl{p2.1} Let $(S, I, R)$ be the solution of \qq{1.3} and $V=S+I$. Assume that $\overline{V}, \ud V, \bar I,\ud I\in C^{2,1}(\overline{\Omega}\times[0,+\infty))$ are positive functions. If $\overline{V}, \ud V, \bar I,\ud I$ satisfy
 \bes\left\{\begin{array}{ll}
  \overline{V}_t-d\Delta\overline{V}\ge aN(x,t)-\gamma\underline I-[\beta+bN(x,t)]\overline{V},\;\;&x\in\Omega,\, t>0,\\[1mm]
  \underline V_t-d\Delta\underline V\le aN(x,t)-\gamma\bar I-(\beta+bN(x,t))\underline V,\;\;&x\in\Omega,\, t>0,\\[1mm]
 \bar I_t-d\Delta\bar I\ge k\mathcal{P}[\bar I](\overline{V}-\bar I)-\gamma\bar I-\beta\bar I-bN(x,t)\bar I,\;\; &x\in\Omega,\,t>0,\\[1mm]
\underline I_t-d\Delta\underline I\le k\mathcal{P}[\underline I](\underline V-\underline I)-\gamma\underline I-\beta\underline I-bN(x,t)\underline I,\;\; &x\in\Omega,\,t>0,\\[2mm]
 \dd\frac{\partial\overline{V}}{\partial \nu}\ge 0\ge\frac{\partial\ud V}{\partial \nu},\;\;\frac{\partial\bar I}{\partial \nu}\ge 0\ge\frac{\partial\ud I}{\partial \nu},\;\; &x\in\partial\Omega,\,t>0,\\[2mm]
 \overline{V}_0(x)\ge V_0(x)\ge\ud V_0(x),\;\;\bar I_0(x)\ge I_0(x)\ge\ud I_0(x),\;\; &x\in\Omega,
 \end{array}\right.
  \lbl{2.3}\ees
then we have that
 \bes
  \ud V\le V\le\overline{V},\;\;\ud I\le I\le\bar I,\;\;\;x\in\overline{\Omega},\,t\ge 0.
  \lbl{2.4}
  \ees
\end{prop}

To prove Proposition \ref{p2.1}, we first give a general maximum principle.

\begin{lem}{\rm(Maximum principle)}\lbl{l2.1} Let $0<T<\infty$ and $d_i>0$ be constants, and $h_{ij}(x,t)$, $a_i(x,t)\in L^\infty(\Omega\times(0,T))$. Assume that $u_i\in C^{2,1}(\overline{\Omega}\times[0,T])$ and satisfy
 \bess\left\{\begin{array}{ll}
u_{i,t}-d_i\Delta u_i\ge\dd\sum_{j=1}^mh_{ij}(x,t)u_j+a_i(x,t)\mathcal{P}[u_i],\;\;
  &(x,t)\in\Omega\times(0,T],\\[1mm]
   \dd\frac{\partial u_i}{\partial \nu}\ge 0,\;\; &(x,t)\in\partial\Omega\times[0,T],\\[2mm]
 u_i(x,0)\ge 0,\;\; &x\in\Omega,\\[2mm]
 i=1,2,\cdots,m.
 \end{array}\right.
  \eess
If
 \bess
 a_i(x,t)\ge 0,\;\;h_{ij}\ge 0 \; \; \mbox{in}\,\; \Omega\times[0,T] \; \; \mbox{for} \; \; j\neq i, \; i, j=1,\dots,m,
 \eess
then we have
 \bess
  u_i(x,t)\ge 0 \; \; \text{in} \,\; \Omega\times[0,T], \; \; i=1,\dots,m,
   \eess
\end{lem}

\begin{proof} This proof is the same as that of \cite[Theorem 1.31]{Wangpara} and \cite[Chapter 3, Theorem 13]{ProWein}. For the convenience of readers, we provide details of the proof. Since $a_i$ and $h_{ij}$ are bounded, there is $\beta>0$ such that
$$\beta>a_i(x,t)+\sum_{j=1}^m h_{ij}(x,t)\;\;\;\text{in} \,\; \Omega\times[0,T]\;\; \mbox{for\, all} \ \ 1\leq i\leq m.$$
Take $\varepsilon>0$, then functions $v_i^\varepsilon=u_i+\varepsilon {\rm e}^{\beta t}$ satisfy
\bess
 v_{i,t}^\varepsilon-d_i\Delta v_i^\varepsilon
&\ge&\dd\sum_{j=1}^mh_{ij}(x,t)u_j
 +a_i(x,t)\mathcal{P}[u_i]+
\beta\varepsilon{\rm e}^{\beta t}\\[1mm]
 &=&\dd\sum_{j=1}^mh_{ij}(x,t)v_j^\varepsilon
 +a_i(x,t)\mathcal{P}[v_i^\varepsilon]+
\varepsilon{\rm e}^{\beta t}\left(\beta-a_i(x,t)-\sum_{j=1}^m h_{ij}(x,t)\right)\\[2mm]
&>&\sum_{j=1}^mh_{ij}(x,t)v_j^\varepsilon+a_i(x,t)\mathcal{P}[v_i^\varepsilon],
 \;\;\;(x,t)\in\Omega\times[0,T],
 \eess
and
 \bess
 \dd\frac{\partial v_i^\varepsilon}{\partial \nu}\ge 0,\;\;\;(x,t)\in\partial\Omega\times[0,T]\eess
as well as
\bess
 v_i^\varepsilon(x,0)=u_i(x,0)+\varepsilon>0,\;\; \;x\in\overline{\Omega}
 \eess
for $i=1,2,\cdots,m$.

Set $v^\varepsilon=(v_1^\varepsilon, \dots, v_m^\varepsilon)$. Thanks to the fact that $v^\varepsilon(x,0)>0$ in $\overline{\Omega}$, there exists $\delta >0$ such that $v^\varepsilon(x,t)>0$ for $x\in\overline{\Omega}$ and $0\leq t\leq \delta$. Define
$$ A=\big\{t:\, t\leq T, ~v^\varepsilon(x,s)>0\, \; \mbox{for all}\,\; x\in \overline{\Omega}, \; 0\leq s\leq t\big\}.$$
Then $t_0={\rm sup}A$ exists and $0<t_0\leq T$.

We shall show that $t_0=T$ and $v^\varepsilon>0$ in $\overline{\Omega}\times[0,T]$. Assume, on the contrary, that our conclusion is not valid, then $v^\varepsilon>0$, i.e., $v_i^\varepsilon>0$ in $\overline\Omega\times[0, t_0)$ for all $1\le i\le m$ and there exists $x_0\in \overline\Omega$ such that $v_j^\varepsilon(x_0, t_0)=0$ for some $j$. Thanks to $a_i(x,t)\ge 0$, it is easy to see that $v_i^\varepsilon$ satisfy
 \bess\left\{\begin{array}{ll}
v_{i,t}^\varepsilon-d_i\Delta v_i^\varepsilon>\dd\sum_{j=1}^mh_{ij}(x,t)v_j^\varepsilon,\;\;  &(x,t)\in\Omega\times(0,t_0],\\[1mm]
   \dd\frac{\partial v_i^\varepsilon}{\partial \nu}\ge 0,\;\; &(x,t)\in\partial\Omega\times[0,t_0],\\[2mm]
 v_i^\varepsilon(x,0)>0,\;\; &x\in\overline{\Omega},\\[2mm]
 i=1,2,\cdots,m.
 \end{array}\right.
  \eess
It then follows by the maximum principle for the parabolic systems (\cite[Lemma 1.30]{Wangpara}) that $v_i^\varepsilon(x,t)>0$ in $\overline{\Omega}\times[0,t_0]$. Which is a contradiction with the fact that $v_j^\varepsilon(x_0, t_0)=0$ for some $j$. Hence, $t_0=T$ and $v^\varepsilon>0$ in $\overline{\Omega}\times[0,T]$.
Letting $\varepsilon \to 0$ we derive $u_i\geq 0$ in $\Omega\times[0,T]$, $i=1,\cdots,m$.
\end{proof}

\begin{proof}[Proof of Proposition \ref{p2.1}] Let $\Tilde{V}_1=\overline{V}-V$, $\Tilde{V}_2=V-\underline V$, $\Tilde{I}_1=\bar I-I$ and $\Tilde{I}_2=I-\underline I$. It is easy to see that
 \bess
 \dd\frac{\partial\Tilde{V}_1}{\partial\nu},\;\frac{\partial \Tilde{V}_2}{\partial\nu},\;\frac{\partial\Tilde{I}_1}{\partial \nu},\;\frac{\partial\Tilde{I}_2}{\partial\nu}\ge 0,\;\;\;x\in\partial\Omega,\,t>0,\eess
and
 \bess
\Tilde{V}_1(x,0),\;\Tilde{V}_2(x,0),\;\Tilde{I}_1(x,0),
 \;\Tilde{I}_2(x,0)\ge0,\;\;\;x\in\Omega.\eess
In addition, we also have
 \bess
 \Tilde{V}_{1t}-d\Delta \Tilde{V}_1&\ge&-(\beta+b N)\Tilde{V}_1+\gamma\Tilde{I}_2,\;\;x\in\Omega,\,t>0,\\[1mm]
 \Tilde{V}_{2t}-d\Delta \Tilde{V}_2&\ge&-(\beta+b N)\Tilde{V}_2+\gamma\Tilde{I}_1,\;\;x\in\Omega,\,t>0,
\eess
as well as
\bess
\Tilde{I}_{1t}-d\Delta\Tilde{I}_1&\ge&k\mathcal{P}[\bar I](\overline{V}-\bar I)-k\mathcal{P}[I](V-I)-(\gamma+\beta+b N)\Tilde{I}_1\\
&=&k\mathcal{P}[\Tilde{I}_1](V-I)-\big(\gamma+\beta+b N+k\mathcal{P}[\bar I]\big)\Tilde{I}_1+k\mathcal{P}[\bar I]\Tilde{V}_1,\;\;x\in\Omega,\, t>0,
 \eess
and
 \bess
\Tilde{I}_{2t}-d\Delta\Tilde{I}_2&\ge&k\mathcal{P}[I](V-I)-k\mathcal{P}[\underline I](\underline V-\underline I)-(\gamma+\beta+b N)\Tilde{I}_2\nonumber\\
&=&k\mathcal{P}[\Tilde{I}_2](V-I)-\big(\gamma+\beta+b N+k\mathcal{P}[\underline I]\big)\Tilde{I}_2+k\mathcal{P}[\underline I]\Tilde{V}_2,\;\;x\in\Omega,\, t>0.
 \eess
Recalling that $V(x,t)-I(x,t)>0$, $\mathcal{P}[\bar I](x,t)\ge 0$ and $\mathcal{P}[\underline I](x,t)\ge 0$, and all coefficients of $\Tilde{V}_i$ and $\Tilde{I}_i$ are bounded in $\Omega\times(0,T)$ for any $T>0$, it follows from Lemma \ref{l2.1} that
$\Tilde{V}_1\ge0, \Tilde{V}_2\ge0, \Tilde{I}_1\ge0$ and $\Tilde{I}_2\ge0$. The conclusion \qq{2.4} is proved. \end{proof}

{\begin{theo}\lbl{t2.2} If $\mathcal{R}_{01}\le1$, then $E_1$ is globally asymptotically  stable for the problem \qq{1.3}.
\end{theo}

\begin{proof} Let $f(t)$ and $g(t)$ be the unique solutions of
 \bess
 u'(t)=(a-\beta-bu)u,\;\; t>0
 \eess
with initial data $f(0)=\max_{\overline{\Omega}}N(x,0)>0$ and $g(0)=\min_{\overline{\Omega}}N(x,0)>0$, respectively. Then
 \bess
f(t)=\frac{a-\beta}{b}+c_1(t){\rm e}^{-(a-\beta)t}>0,\;\;\;
g(t)=\frac{a-\beta}{b}+c_2(t){\rm e}^{-(a-\beta)t}>0,\;\;\forall\;t\ge 0,
 \eess
where $c_1(t)$ and $c_2(t)$ are bounded and continuous functions.
The comparison principle yields
 \bes
g(t)\le N(x,t)\le f(t),\;\;\forall\;x\in\Omega,\,t\ge0.
\nonumber\ees
Let $(\overline{V}(t),\underline V(t),\bar I(t),\underline I(t))$ be the unique positive solution of the initial value problem
\bes\left\{\begin{array}{ll}
\overline{V}'(t)=af(t)-\gamma\underline I-(\beta+bg(t))\overline{V},\;\; &t>0\\[1mm]
\underline V'(t)=ag(t)-\gamma\bar I-(\beta+bf(t))\underline V,\;\; &t>0\\[1mm]
\bar I'(t)=k\bar I(\overline{V}-\bar I)-\gamma\bar I-\beta\bar I-bg(t)\bar I,\;\; &t>0\\[1mm]
\underline I'(t)=k\underline I(\underline V-\underline I)-\gamma\underline I-\beta\underline I-bf(t)\underline I,\;\; &t>0\\[1mm]
\overline{V}(0)=\dd\max_{\overline{\Omega}}V_0(x)>0,\;\underline V(0)=\dd\min_{\overline{\Omega}}V_0(x)>0,\;\\[1mm]
\bar I(0)=\dd\max_{\overline{\Omega}}I_0(x)>0,
\;\underline I(0)=\dd\min_{\overline{\Omega}}I_0(x)>0.
\end{array}\right.\lbl{2.5}
\ees
Then $\overline{V}(t),\underline V(t),\bar I(t),\underline I(t)>0$ for all $t\ge 0$.
It is clear that $(\overline{V}(t),\underline V(t),\bar I(t),\underline I(t))$ satisfies \qq{2.3} since $\int_{\Omega}P(x,y){\rm d}x=1$ for any $y\in\Omega$. Hence, \qq{2.4} holds by Lemma \ref{l2.1}.

Noticing that
 \bess
\overline{V}'(t)\le af(t)-(\beta+bg(t))\overline{V},
 \eess
and $\dd\lim_{t\to\infty}f(t)=\lim_{t\to\infty}g(t)=(a-\beta)/b$. It follows that
 \bess
\limsup_{t\to\infty }\overline{V}(t)\le(a-\beta)/b.
 \eess

Rewrite the equation of $\bar I$ as
 $$\bar I'(t)=[k\overline{V}-\gamma-\beta-bg(t)]\bar I-k\bar I^2.$$
Since $\mathcal{R}_{01}\le 1$, we have
 \bess
\limsup_{t\to\infty}[k\overline{V}-\gamma-\beta-bg(t)]\le k(a-\beta)/b-(a+\gamma)\le 0.
 \eess
Consequently, $\dd\limsup_{t\to\infty }\bar I(t)\le0$. This combined with \qq{2.4} allows us to derive
 \bess
 \lim_{t\to\infty } I(x,t)=0\;\;\;\text{uniformly in}\;\;\overline{\Omega}.\eess
Thus, by the equations of $R(x,t)$,
 \bess
 \lim_{t\to\infty } R(x,t)=0\;\;\;\text{uniformly in}\;\;\overline{\Omega}.\eess
Using the equations of $\overline{V}(t)$ and $\ud V(t)$ we can obtain that
 $$\lim_{t\to\infty }\overline{V}(t)=\lim_{t\to\infty }\ud V(t)=(a-\beta)/b,$$
which implies
 $$\lim_{t\to\infty } V(x,t)=(a-\beta)/b\;\;\;\text{uniformly in}\;\;\overline{\Omega}$$
by \qq{2.4}. Therefore,
 $$\lim_{t\to\infty}S(x,t)=\lim_{t\to\infty}V(x,t)-\lim_{t\to\infty}I(x,t)=(a-\beta)/b\;\;\;\text{uniformly in}\;\;\overline{\Omega}.$$
Using the estimate \eqref{2.1} we can also get
 $$\lim_{t\to\yy}(S, I, R)=((a-\beta)/b, 0, 0)=E_1\;\;\;\text{in}\;\; [C^1(\overline{\Omega})]^3.$$
The proof is complete.  \end{proof}

\begin{theo}\lbl{t2.3} If $\mathcal{R}_{01}>1$ and $a>\gamma$, then $E_2$ is globally asymptotically stable for the system \eqref{1.3}.
\end{theo}

\begin{proof} Obviously,
\bes\left\{\begin{array}{ll}
aN_*-\gamma I_*-(\beta+bN_*)V_*=0,\\[1mm]
kV_*-kI_*-\gamma -\beta-bN_*=0.
\end{array}\right.\lbl{2.6}
\ees
where $V_*=S_*+I_*$. Let $(\overline{V},\underline V,\bar I,\underline I)$ be the solution of \qq{2.5}. Then \qq{2.4} holds. It suffices to show that
 \bes
 \lim_{t\to\yy}\overline{V}=V_*,\;\;\lim_{t\to\yy}\underline V=V_*,\;\;
\lim_{t\to\yy}\bar I=I_*,\;\;\lim_{t\to\yy}\underline I=I_*.
 \lbl{2.7}\ees
Using \qq{2.5} and \qq{2.6}, and the expressions of $f(t)$ and $g(t)$, we have
\bess\left\{\begin{array}{ll}
\overline{V}'(t)=-\gamma(\underline I-I_*)-(\beta+bN_*)(\overline{V}-V_*)+[ac_1(t)-bc_2(t)\overline{V}]{\rm e}^{-(a-\beta)t},\;\; &t>0\\[1mm]
\underline V'(t)=-\gamma(\bar I-I_*)-(\beta+bN_*)(\underline V-V_*)+[ac_2(t)-bc_1(t)\underline V]{\rm e}^{-(a-\beta)t},\;\; &t>0\\[1mm]
\bar I'(t)=k\bar I\left[(\overline{V}-V_*)-(\bar I-I_*)-\frac{b}{k}c_2(t){\rm e}^{-(a-\beta)t}\right],\;\; &t>0\\[1mm]
\underline I'(t)=k\underline I\left[(\underline V-V_*)-(\underline I-I_*)-\frac{b}{k}c_1(t){\rm e}^{-(a-\beta)t}\right],\;\; &t>0\\[1mm]
\overline{V}(0)=\dd\max_{\overline{\Omega}}V_0(x),\;\underline V(0)=\dd\min_{\overline{\Omega}}V_0(x),\\[1mm]\;
\bar I(0)=\dd\max_{\overline{\Omega}}I_0(x),\;\underline I(0)=\dd\min_{\overline{\Omega}}I_0(x).
\end{array}\right.
 \eess
Define a Lyapunov functional as follows
\[F(t)=\frac{1}{2}(\overline{V}-V_*)^2+\frac{1}{2}(\underline V-V_*)^2+\frac{\gamma}{k}\int_{I_*}^{\bar I}\frac{s_1-I_*}{s_1}{\rm d} s_1+\frac{\gamma}{k}\int_{I_*}^{\underline I}\frac{s_2-I_*}{s_2}{\rm d} s_2+\lambda {\rm e}^{-(a-\beta)t}.
\]
where $\lambda>0$ to be determined later. Then $F(t)\ge 0$ for all $t\ge 0$.  Careful computations lead to
 \bes
 \frac{{{\rm d}}F}{{{\rm d}}t}&=&(\overline{V}-V_*)\overline{V}'+(\underline V-V_*)\underline V'+\frac{\gamma}{k}\frac{\bar I-I_*}{\bar I}\bar I'+\frac{\gamma}{k}\frac{\underline I-I_*}{\underline I}\underline I'-
\lambda(a-\beta){\rm e}^{-(a-\beta)t}\nonumber\\[1mm]
&=&-a(\overline{V}-V_*)^2-\gamma(\overline{V}-V_*)(\underline I-I_*)+(ac_1(t)-bc_2(t)\overline{V}){\rm e}^{-(a-\beta)t}(\overline{V}-V_*)\nonumber\\[1mm]
 &&-a(\underline V-V_*)^2-\gamma(\underline V-V_*)(\bar I-I_*)+(ac_2(t)-bc_1(t)\underline V){\rm e}^{-(a-\beta)t}(\underline V-V_*)
\nonumber\\[1mm]
&&+\gamma(\overline{V}-V_*))(\bar I-I_*)-\gamma(\bar I-I_*)^2-\frac{\gamma b}{k}c_2(t){\rm e}^{-(a-\beta)t}(\bar I-I_*)
\nonumber\\[1mm]
&&+\gamma(\underline V-V_*)(\underline I-I_*)-\gamma(\underline I-I_*)^2-\frac{\gamma b}{k}c_1(t){\rm e}^{-(a-\beta)t}(\underline I-I_*)-\lambda (a-\beta){\rm e}^{-(a-\beta)t}\nonumber\\[1mm]
&\le&-\left[a(\overline{V}-V_*)^2+a(\underline V-V_*)^2+\gamma(\bar I-I_*)^2
+\gamma(\ud I-I_*)^2-\gamma(\overline{V}-V_*)(\bar I-I_*)\right.\nonumber\\[1mm]
 &&\left.-\gamma(\underline V-V_*)(\underline I-I_*)+\gamma(\overline{V}-V_*)(\underline I-I_*)+\gamma(\underline V-V_*)(\bar I-I_*)\right]
 \label{2.8}\ees
provided that $\lambda\gg1$ as the functions $c_1(t), c_2(t), \overline{V}(t), \ud V(t), \bar I(t)$ and $\ud I(t)$ are bounded. It is easy to see that when
$a>\gamma$, then there exists $\ep>0$ such that
 \bess
 &&a(\overline{V}-V_*)^2+a(\underline V-V_*)^2+\gamma(\bar I-I_*)^2
+\gamma(\ud I-I_*)^2-\gamma(\overline{V}-V_*)(\bar I-I_*)\nonumber\\[1mm]
 &-&\gamma(\underline V-V_*)(\underline I-I_*)+\gamma(\overline{V}-V_*)(\underline I-I_*)+\gamma(\underline V-V_*)(\bar I-I_*)\\[1mm]
 &\ge&\ep[(\overline{V}-V_*)^2+(\underline V-V_*)^2+(\bar I-I_*)^2
+(\ud I-I_*)^2].
 \eess
Thus, by \eqref{2.8}, we have
 \bess
 \frac{{{\rm d}}F}{{{\rm d}}t}\le-\ep[(\overline{V}-V_*)^2+(\underline V-V_*)^2+(\bar I-I_*)^2+(\ud I-I_*)^2].
 \eess
By use of \cite[Lemma 1.1]{Waml18}, we have
  $$\lim_{t\to\yy}[(\overline{V}-V_*)^2+(\underline V-V_*)^2+(\bar I-I_*)^2+(\underline I-I_*)^2]=0,$$
which implies \qq{2.7}. Combining with \qq{2.4} and \qq{2.7}, we have
\bes\dd\lim_{t\to\yy}V(x,t)=V_*,\;\;\dd\lim_{t\to\yy}I(x,t)=I_*\;\;\;\text{uniformly in}\;\;\overline{\Omega}.\nonumber\ees
It follows that
 $$\lim_{t\to\yy}S(x,t)=\lim_{t\to\yy}V(x,t)-\lim_{t\to\yy}I(x,t)=V_*-I_*=S_*
  \;\;\;\text{uniformly in}\;\; \overline{\Omega}.$$
Noticing that $\dd\lim_{t\to\yy}N(x,t)=N_*$ uniformly in $\overline{\Omega}$, we have immediately that
 $$\lim_{t\to\yy}R(x,t)=\lim_{t\to\yy}N(x,t)-\lim_{t\to\yy}V(x,t)=N_*-V_*=R_*\;\;\;\text{uniformly in}\;\; \overline{\Omega}.$$
Using the estimate \eqref{2.1} we can also get
 $$\lim_{t\to\yy}(S, I, R)=(S_*, I_*,R_*)\;\;\;\text{in}\;\; [C^1(\overline{\Omega})]^3.$$
The proof is complete.
\end{proof}

\section{The dynamical properties of the problem \eqref{1.4}}
	\setcounter{equation}{0} {\setlength\arraycolsep{2pt}

In this section we investigate the dynamical properties of \eqref{1.4}. We first prove the well-posedness and positivity of solutions. Then we study the positive equilibrium solutions, including the necessary and sufficient conditions for the existence, and the uniqueness in some special case. Throughout this section we always assume that hypotheses {\bf(P1)} and {\bf(D)} hold.

\subsection{Existence, uniqueness and positivity of the solution of \eqref{1.4}}

\begin{theo}\lbl{t3.1} The problem \eqref{1.4} has a unique global solution $(S(x,t), I(x,t), R(x,t))$ and
 \bes
 0<S(x,t), I(x,t), R(x,t)\le\max\kk\{\dd\max_{\overline{\Omega}}\{S_0(x)+I_0(x)\},\,(a-\beta)/b\rr\}
\;\;\;\text{for} \;x\in\overline{\Omega},\;t>0.
\nonumber\ees
Moreover, for any given $0<\alpha<1$ and $\tau>0$, $S,\,I,\,R\in C^{1+\alpha,\,(1+\alpha)/2}(\overline\Omega\times[\tau, \infty))$, and there exists a constant $C(\tau)$ such that
\begin{eqnarray}
\|S,\,I,\,R\|_{C^{1+\alpha,\,(1+\alpha)/2}(\overline\Omega\times[\tau, \infty))}\leq C(\tau).
\nonumber\end{eqnarray}
\end{theo}

\begin{proof}
The proof of this theorem is akin to that of Theorem \ref{t2.2} and we omit the details here.
\end{proof}

\subsection{Nonnegative equilibrium solutions of \eqref{1.4}}

The equilibrium problem of \eqref{1.4} reads as
\bes
 \left\{\begin{array}{lll}
-d\Delta S=aN-\beta S-bNS-k\mathcal{P}[I]S,\;\; &x\in\Omega, \\[1mm]
 -d\Delta I=k\mathcal{P}[I]S-\gamma I-\beta I-b NI,\ &x\in\Omega,\\[1mm]
 -d\Delta R=\gamma I-\beta R-bNR,\ &x\in\Omega,\\[1mm]
  S=I=R=0, \ &x\in\partial\Omega.
 \end{array}\right.
 \lbl{3.1}\ees
In this part, we shall use the theory of topological degree in cones to study the positive solutions of \eqref{3.1}, including the necessary and sufficient conditions for the existence of positive solutions, and the uniqueness of  positive solutions in some special case.

To use the maximum principle and ensure the following expression \qq{3.10} it is required that $P(x,y)=P(y,x)>0$ for all $x,y\in\Omega$. To prove the following Proposition \ref{p3.3}, the Hopf boundary lemma is required, so we need the solution $u$ of the following boundary value problem with nonlocal terms:
 \bes\left\{\begin{array}{lll}
-d\Delta u-a(x)\mathcal{P}[u]+b(x)u=h(x),\;\;&x\in\Omega\;,\\[1mm]
u=0,\;\;&x\in\partial\Omega
\end{array}\right.
 \lbl{3.2}\ees
belongs to $C^2(\Omega)\cap C^1(\overline{\Omega})$ at least when $a,b,h\in C^\alpha(\overline{\Omega})$. To this aim, in this subsection, we always assume
 \bess
 P\in C^\alpha(\overline{\Omega}\times\overline{\Omega}),\;\;\;\text{and}\;\;
 P(x,y)=P(y,x)>0\;\;\text{ for\; all}\;\; x,y\in\Omega.
 \eess
Under this condition, we see that the solution $u$ of \qq{3.2} belongs to $W^2_p(\Omega)$ for all $p>1$ by the $L^p$ theory of elliptic equations. So, $u\in C^{1+\alpha}(\overline{\Omega})$ by the embedding theorem. And in turn, $\mathcal{P}[u]\in C^\alpha(\overline{\Omega})$. Consequently, $u\in C^{2+\alpha}(\overline{\Omega})$ by the Schauder theory of elliptic equations.

On the other hand, if $a(x), h(x)\ge 0$, and $0\le u\in C^1(\overline{\Omega})\cap
 C^2(\Omega)$ satisfies \qq{3.2}. Then we have that
\bess\left\{\begin{array}{lll}
-d\Delta u+(b(x)+M)u=h(x)+Mu+a(x)\mathcal{P}[u]\ge 0,\;\;&x\in\Omega\;,\\[1mm]
u=0,\;\;&x\in\partial\Omega
\end{array}\right.
 \eess
and $b(x)+M>0$ for some positive constant $M$. It is followed, by the strong maximum principle and the Hopf boundary lemma of the elliptic equations with only local terms, that either $u\equiv 0$, or $u>0$ in $\Omega$ and $\frac{\partial u}{\partial \nu}<0$ on $\partial\Omega$. This shows that the Hopf boundary lemma holds for the problem \qq{3.2}.

Apparently, $\boldsymbol{0}=(0,0,0)$ is a trivial solution of \qq{3.1}.

For $r\in L^\infty(\Omega)$, we define the principal eigenvalue of
 \bess\left\{\begin{array}{lll}
-d\Delta\phi+r(x)\phi=\lambda\phi,\;\;&x\in\Omega,\\[1mm]
\phi=0,\;\;&x\in\partial\Omega
\end{array}\right.
 \eess
by $\lambda_1(r)$. When $r\equiv 0$, we define simply $\lambda_1(0)=\lambda_0$. It is well known that
 $$\lambda_1(r+c)=\lambda_1(r)+c$$
when $c$ is a constant.

It is well-known that the problem
 \bes
 \left\{\begin{array}{lll}
-d\Delta N=(a-\beta)N-b N^2,\;\; &x\in\Omega, \\[1mm]
 N=0, \ &x\in\partial\Omega
 \lbl{3.3}\end{array}\right.
 \ees
has a positive solution, defined by $\tilde N$, if and only if $a-\beta>\lambda_0$, i.e., $\lm_1(\beta-a)<0$, and $\tilde N$ is unique, non degenerate and globally asymptotically stable when it exists. Apparently, $(\tilde N, 0, 0)$ is a semi-trivial nonnegative solution of \qq{3.1}.

In Subsection \ref{s3.2.1} we state some abstract results about topological degree in cones and give the estimates of positive solutions of \qq{3.1}. For positive solutions of \qq{3.1}, the necessary and sufficient conditions for the existence, and the uniqueness in some special case are obtained in Subsection \ref{s3.2.2}.

Now we state a well-known conclusion , which will be used constantly later.

\begin{prop}\lbl{p3.1}
The strong maximum principle holds for the operator $-d\Delta +q(x)$ in $\Omega$ with the homogeneous Dirichlet boundary conditions if and only if $\lm_1(q)>0$, and the principal eigenvalue $\lm_1(q)$ is strict increasing in $d$ and $q(x)$.
\end{prop}

\subsubsection{Abstract results}\lbl{s3.2.1}

To save spaces, we set $u=(S, I, R)=(u_1,u_2,u_3)$ and $f=(f_1, f_2, f_3)$ with
 \bes\left\{\begin{array}{lll}
 f_1(u)&=&a(S+I+R)-\beta S-b(S+I+R)S-k\mathcal{P}[I]S,\\[1mm]
 f_2(u)&=&k\mathcal{P}[I]S-\gamma I-\beta I-b(S+I+R)I,\\[1mm]
 f_3(u)&=&\gamma I-\beta R-b(S+I+R)R.
 \end{array}\right.
 \nonumber\ees
Clearly, $f_i(u)\big|_{u_i=0}\ge 0$ for all $u\ge 0$. The problem \qq{3.1} can be written as
 \bes
 \left\{\begin{array}{lll}
 \mathscr{L}u=f(u),\;\; &x\in\Omega, \\[1mm]
 u=0, \ &x\in\partial\Omega,
 \lbl{3.4}\end{array}\right.
 \ees
where
 $$\mathscr{L}={\rm diag}(-d\Delta,-d\Delta,-d\Delta).$$

For the enough large positive constant $M$, we define
 \bess
  E&=&X^3 \;\;\text{with}\;\; X=\{z\in C^1(\overline{\Omega}):\, z=0\,\;\text{on}\,\,\partial\Omega\},
  \nonumber\\
  W&=&K^3\;\;\text{with}\;\; K=\{z\in X:\, z\geq 0\,\;\text{in}\,\,\Omega\}, \nonumber\\
  F(u)&=&\big(M+\mathscr{L}\big)^{-1}\big(f(u)+Mu\big).
  \eess
Then $u\in W$ is a solution of \qq{3.4} if and only if $F(u)=u$. For $u\in W$, we define
 \bess
 W_u&=&\{v\in E:\, \exists\; r>0\;\; \text{s.t.}\;\;u+tv\in
W,\,\forall\; 0\leq t\leq r\},\\
 S_u&=&\{v\in\overline{W}_u:-v\in\overline{W}_u\}.
 \eess
\begin{theo}{\rm(\cite[Corollary 3.1]{WangYao})}\lbl{t3.2} Assume that $u\in W$ is an isolated fixed point of $F$, and the problem
 \bes
 F'(u)\phi=\phi,\;\;\phi\in \overline{W}_u
 \lbl{3.5}
 \ees
has only the zero solution. Then the following statements hold.
\begin{enumerate}[leftmargin=3em]
\item[{\rm(1)}]\, If the eigenvalue problem
   \bes
 F'(u)\phi=\lm\phi,\;\;\phi\in\overline{W}_u\setminus S_u
 \lbl{3.6}
 \ees
has an eigenvalue $\lm>1$, then ${\rm index}_W(F,u)=0$;

\item[{\rm(2)}]\, If all eigenvalues of the eigenvalue problem
   \bes
 F'(u)\phi-\lm\phi\in S_u,\;\;\phi\in\overline{W}_u\setminus S_u
 \nonumber
 \ees
are less than $1$, then
 $${\rm index}_W(F,u)={\rm index}_E(F,u)=(-1)^{\beta},$$
where $\beta$ is the sum of algebraic multiplicities of all eigenvalues
of $F'(u)$ that are greater than one.
 \end{enumerate}
\end{theo}

In order to calculate the topological degree of the operator $I-F$, we first study the following auxiliary problem
 \bes
 \left\{\begin{array}{lll}
 \mathscr{L}u=\tau f(u),\;\; &x\in\Omega, \\[1mm]
 u=0, \ &x\in\partial\Omega
 \end{array}\right.\lbl{3.7}\ees
with $0\leq\tau\leq 1$. It is easy to show that any non-negative solution $u_\tau$ of \eqref{3.7} satisfies
 \bess
 \|u_\tau\|_{L^\infty(\Omega)}\leq (a-\beta)/b.
 \eess
Applying the $L^p$ theory and embedding theorem to \eqref{3.7}, there exists a positive constant $C$ such that any non-negative solution $u_\tau$ of \eqref{3.7} satisfies
 \bes
 \|u_\tau\|_{C^1(\overline{\Omega})}\leq C.
  \nonumber\ees
 Set
 \begin{eqnarray*}
 {\cal O}=\{u\in W:\, \|u\|_{C^1(\overline{\Omega})}<C+1\}.
  \end{eqnarray*}
We can enlarge $M$ such that
 \bess
 \tau f_i(u)+Mu_i\geq 0,\;\;\;\forall\, u\in {\cal O},\;0\leq \tau\leq 1.
 \eess
Define
 $$ F_\tau(u)=(\mathscr{L}+M)^{-1}\left(\tau f(u)+Mu\right),\;\;u\in E,\; 0\leq \tau\leq 1,$$
and $F=F_1$. It is easily to obtain from the homotopy invariance of the topological degree that
 \bes
 {\rm deg}_W(I-F,{\mathcal O})={\rm deg}_W(I-F_0,{\mathcal O})
 ={\rm index}_W(F_0,\boldsymbol{0})=1.
 \lbl{3.8}\ees

\subsubsection{The existence of positive solutions of  \qq{3.1}}\lbl{s3.2.2}

Assume $\lm_1(\beta-a)<0$. Then the non-negative trivial solutions of \eqref{3.1} are $\boldsymbol{0}=(0,0,0)$ and $(\tilde N,0,0)$.
Define $E$ and $W$ as above. It is easy to show that
 \[\overline{W}_{\boldsymbol{0}}=W,\;\;S_{\boldsymbol{0}}=\{\boldsymbol{0}\},
 \;\;\overline{W}_{(S^*,0,0)}=X\times K\times K,\;\;S_{(S^*,0,0)}=X\times\{0\}\times\{0\}.\]
The direct calculation yields, we have
\[F'(u)=(M-d\Delta)^{-1}\big(f'(u)+M\big),\]
where
\bes
f'(u)=\left(\begin{matrix}a-\beta-bN-bS-k\mathcal{P}[I] \;\;& a-bS-k\mathcal{P}[\cdot]S \;\;&a-bS\\[1mm]
 k\mathcal{P}[I]-bI\;\; &k\mathcal{P}[\cdot]S-(\gamma+\beta)-bN-bI \;\;&-bI\\[1mm]-bR \;\;&\gamma-bR \;\;&-\beta-bN-bR\end{matrix}\right).
 \nonumber
 \ees

 \begin{lem}\lbl{l3.1} If $\lm_1(\beta-a)<0$ then ${\rm index}_W(F, \boldsymbol{0})=0$.
\end{lem}

The proof of this lemma is the same as that of \cite[(3.16)]{WangYao}, and we omit the details.

For future applications, we will first discuss the eigenvalue problem with non local terms:
  \bes\left\{\begin{array}{lll}
-d\Delta\phi-a(x)\mathcal{P}[\phi]+b(x)\phi=\lambda\phi,\;\;&x\in\Omega\;,\\[1mm]
\phi=0,\;\;&x\in\partial\Omega,
\end{array}\right.
 \lbl{3.9}\ees
where $a,b\in C^\alpha(\overline{\Omega})$ in $\Omega$. By Krein-Rutman theorem, one can easily show that \eqref{3.9} has a unique principal eigenvalue denoted by $\lambda_1(\mathcal{P}, \Omega, a, b)$. Since $P(x,y)$ is symmetric, we have
 \begin{equation}\label{3.10}
\lambda_1(\mathcal{P}, \Omega, a, b)=\inf_{\substack{\phi\in H_0^1(\Omega)\\		 \|\phi\|_2=1}}\left(d\int_\Omega|\nabla\phi|^2{\rm d}x-\iint_{\Omega\times\Omega}a(x)P(x,y)\phi(y)\phi(x)\,{\rm d}y{\rm d}x+
\int_\Omega b(x)\phi^2{\rm d}x\right).\qquad
 \end{equation}

\begin{prop}\label{p3.2} Let $\lambda_1(\mathcal{P}, \Omega, a, b)$ be the principal eigenvalue of \eqref{3.9}.
 \begin{enumerate}
\item [\rm (i)] $\lambda_1(\mathcal{P}, \Omega, a, b)$ is continuous in both $a(x)$ and $b(x)$, and is strictly decreasing and increasing in $a(x)$ and $b(x)$, respectively;
\item [\rm (ii)] $\lambda_1(\mathcal{P}, \Omega, a, b)$ is strictly decreasing with respect to $\Omega$ when $a(x)>0$;
\end{enumerate}
\end{prop}

\begin{proof} The proofs (i) and (ii) are standard and we omit the details. The interested readers can refer to \cite[Section 2.2.3]{WangPang}.
When $\Omega=(l_1,l_2)$ is an interval (the one dimensional case) and $a>0$ is a constant, some properties of $\lambda_1(\mathcal{P}, \Omega, a, b)$ were given in \cite{HWzamp19}.
\end{proof}

Define an operator $(\mathscr{L}_P, D)$ by
 \bess\left\{\begin{array}{ll}
 \mathscr{L}_Pu=-d\Delta u-a(x)\mathcal{P}[\cdot]u+b(x)u,\;\;&x\in\Omega,\\[1mm]
 Du=u,\;\;&x\in\partial\Omega.
 \end{array}\right.\eess

 \begin{defi}\label{d2.1} The operator
$(\mathscr{L}_P, D)$ is said to have the strong maximum principle property if, for any function $u\in C^1(\overline\Omega)\cap C^2(\Omega)$, from
\begin{eqnarray*}
 -d\Delta u-a(x)\mathcal{P}[\cdot]u+b(x)u\geq 0,\;\;u\not\equiv 0\;\;\;\text{in}\;\;\Omega,\;\;\;
 u\geq 0\;\;\;\text{on}\;\;\partial\Omega,
 \end{eqnarray*}
one can conclude $u>0$ in $\Omega$.
\end{defi}

\begin{defi}\label{d2.2} A function $u\in
C^1(\overline\Omega)\cap C^2(\Omega)$ is said to be an upper
solution of the operator $(\mathscr{L}_P, D)$ if
 \bess
 -d\Delta u-a(x)\mathcal{P}[\cdot]u+b(x)u\geq 0\;\;\;\text{in}\;\;\Omega,\;\;\;
 u\geq 0\;\;\;\text{on}\;\;\partial\Omega.\eess
It is called a strict upper solution if it is an upper solution but not a solution.
\end{defi}

\begin{prop}\lbl{p3.3} Assume that $a(x)\ge 0$ in $\Omega$. Then the following statements are equivalent{\rm:}
 \begin{enumerate}[leftmargin=3em]
 \item[{\rm(1)}]\, $(\mathscr{L}_P, D)$ has the strong maximum principle property{\rm;}
\item[{\rm(2)}]\, $(\mathscr{L}_P, D)$ has a strict upper solution which is positive in $\Omega${\rm;}
\item[{\rm(3)}]\, the principal eigenvalue
$\lambda_1(\mathcal{P}, \Omega, a, b)>0$.
 \end{enumerate}
\end{prop}

The proof of Proposition \ref{p3.3} is the same as that of \cite[Theorem 2.8]{WangPang} (see also \cite[Theorem 2.4]{Du06}, \cite[Theorem 2.5]{Lop96}). The details are omitted here.

\begin{lem}\lbl{l3.2} Let $\lm_1(\mathcal{P},\tilde N)$ be the principal eigenvalue of
 \bes\left\{\begin{array}{lll}
-d\Delta\phi-k\tilde N\mathcal{P}[\phi]+(\gamma+\beta+b\tilde N)\phi=\lambda\phi,\;\;&x\in\Omega\;,\\[1mm]
\phi=0,\;\;&x\in\partial\Omega.
\end{array}\right.
 \nonumber\ees
If $\lm_1(\mathcal{P},\tilde N)<0$, then ${\rm index}_W(F, (\tilde N,0,0))=0$.
\end{lem}

\begin{proof} Take $u=(\tilde N,0,0)$ in problems \qq{3.5} and \qq{3.6}. Then
\[F'(\tilde N,0, 0)=(M-d\Delta)^{-1}\left(\begin{matrix}
a-\beta-2b\tilde N+M \;\;&a-(b+k\mathcal{P}[\cdot])\tilde N \;\;&a-b\tilde N\\[1mm]
 0\;\; &k\mathcal{P}[\cdot]\tilde N-(\gamma+\beta+b\tilde N)+M\;\; &0\\[1mm]0\;\; &\gamma\;\; &-\beta-b\tilde N+M\end{matrix}\right).\]

We first prove that \qq{3.5} has only the zero solution. In fact, let $(\varphi_1, \varphi_2,\varphi_3)\in \overline{W}_{(\tilde N,0,0)}=X\times K\times K$ be a solution of \qq{3.5}. If $\varphi_2\not\equiv 0$, then
 \bess\left\{\begin{array}{lll}
-d\Delta\varphi_2-k\tilde N\mathcal{P}[\varphi_2]+(\gamma+\beta+b\tilde N)\varphi_2=0,\;\;&x\in\Omega\;,\\[1mm]
\varphi_2=0,\;\;&x\in\partial\Omega.
\end{array}\right.
 \eess
This shows that $\lm_1(\mathcal{P},\tilde N)=0$ as $\varphi_2\ge 0, \not\equiv 0$. This is a contradiction and so $\varphi_2=0$. Furthermore, applying the strong maximum principle to the third equation of \eqref{3.5} we deduce $\varphi_3\equiv 0$. Therefore, $\varphi_1$ satisfies
 \begin{eqnarray*}
\left\{\begin{array}{ll}
-d\Delta \varphi_1+(\beta-a+2b\tilde N)\varphi_1=0,\;\;&x\in\Omega,\\[1mm]
\varphi_1=0,\;\;&x\in\partial\Omega.
 \end{array}\right.
 \end{eqnarray*}
Noticing that $\tilde N$ is the positive solution of \qq{3.3}, we have
  $$\lm_1(\beta-a+b\tilde N)=0,$$
and then $\lm_1(\beta-a+2b\tilde N)>0$. Consequently, $\varphi_1\equiv 0$ by the strong maximum principle.

Now we study the eigenvalue problem
  \begin{eqnarray}
 {\cal A}z:=(M-d\Delta)^{-1}\big[Mz+k\tilde N\mathcal{P}[z]-(\gamma+\beta+b\tilde N)z\big]=\lambda z,\;\; z\in K.
  \label{3.11}\end{eqnarray}
Since $P(x,y)=P(y,x)$ for all $x,y\in\Omega$, similar to the proof of \cite[Theorem 2.25]{WangPang}, we can show that $\lm_1(\mathcal{P},\tilde N)<0\; (>0,\,=0)$ is equivalent to $r({\mathcal A})>1\; (<1,\,=1)$.

Thanks to $\lm_1(\mathcal{P},\tilde N)<0$, we see that $r\left({\cal A}\right)>1$ is the principal  eigenvalue of \qq{3.11}. Let $\phi_2>0$ be the eigenfunction corresponding to $r\left({\cal A}\right)$. Noticing that $\lambda:=r\left({\cal A}\right)>1$,
by the monotonicity of $\lambda_1(q)$, we know
 \[\lambda_1\left(M-\frac{1}{\lambda}(M-\beta-b\tilde N)\right)>\lambda_1(\beta+b\tilde N)>\lambda_1(\beta-a+b\tilde N)=0.\]
As $\phi_2>0$, it follows that the problem
 \bess\left\{\begin{array}{ll}
 -d\Delta\phi_3+\big[M-\frac{1}{\lambda}(M-\beta-b\tilde N)\big]\phi_3=\frac{\gamma\phi_2}{\lambda},\;\;&x\in\Omega,\\
  \phi_3=0,&x\in\partial\Omega
  \end{array}\right.\eess
has a unique positive solution $\phi_3$. Noticing that
 \[\lambda_1\left(M-\frac{1}{\lambda}(a-\beta-2b \tilde N+M)\right)>\lambda_1(\beta-a+2b \tilde N)>0.\]
We can easily see that the problem
 \bess\left\{\begin{array}{ll}
-d\Delta\phi_1+\big[M-\frac{1}{\lambda}(a-\beta-2b \tilde N+M)\big]\phi_1=\frac{1}{\lambda}\big\{\big[a-(b+k\mathcal{P}[\cdot])\tilde N\big]\phi_2+(a-b \tilde N)\phi_3\big\},\;\;&x\in\Omega,\\
  \phi_1=0,&x\in\partial\Omega
  \end{array}\right.\eess
has a unique solution $\phi_1$. Therefore, $r\left({\cal A}\right)>1$ is an eigenvalue of \qq{3.6} and $\phi=(\phi_1, \phi_2, \phi_3)\in\overline{W}_{(\tilde N,0,0)}\setminus S_{(\tilde N,0,0)}$ is the corresponding eigenfunction. It follows that ${\rm index}_W(F, (\tilde N,0,0))=0$ by Theorem \ref{t3.2}.
\end{proof}

\begin{theo}\lbl{t3.3} \begin{enumerate}
\item [\rm (i)] The problem \qq{3.1} has a  positive solution $(S,I,R)$ if and only if
 \bess
\lm_1(\beta-a)<0 \;\;\;and\;\; \lm_1(\mathcal{P},\tilde N)<0.
 \eess
\item [\rm (ii)] Let $N=1$ {\rm(}the one dimension case{\rm)}, $(S_1,I_1,R_1)$ and $(S_2,I_2,R_2)$ be positive solutions of \eqref{3.1}. If $\mathcal{P}[I_1]\ge \mathcal{P}[I_2]$ or $R_1\ge R_2$ in $\Omega$, then $(S_1,I_1,R_1)\equiv (S_2,I_2,R_2)$.
\end{enumerate}
\end{theo}

\begin{proof} (i) Let's prove the sufficiency firstly. We have known that ${\rm deg}_W(I-F,{\mathcal O})=1$ by \qq{3.8}, ${\rm index}_W(F,\boldsymbol{0})=0$ by Lemma \ref{l3.1}, ${\rm index}_W(F, (\tilde N, 0,0))=0$ by Lemma \ref{l3.2}. As
  \bess
  {\rm index}_W(F,\boldsymbol{0})+{\rm index}_W(F, (\tilde N, 0,0))
  \not={\rm deg}_W(I-F,{\mathcal O}),\eess
the problem \eqref{3.1} has at least one positive solution.

In the following we shall prove the necessity. Let $(S,I,R)$ be a positive solution of \qq{3.1}. Set $N=S+I+R$, it is easy to see that $N$ is a positive solution of \qq{3.3}. Thus $\lm_1(\beta-a)<0$, and $N=\tilde N$ by the uniqueness of positive solutions of \qq{3.3}. Furthermore, by the equations of $I$ in \qq{3.1},
 \bess\left\{\begin{array}{lll}
-d\Delta I=k\mathcal{P}[I]S-(\gamma +\beta+b N)I<k\tilde N\mathcal{P}[I]-(\gamma +\beta+b \tilde N)I\;\;&x\in\Omega,\\[2mm]
I=0\;\;&x\in\partial\Omega.
\end{array}\right.
 \eess
On the basis of $I>0$ and $P(x,y)=P(y,x)$ for all $x,y\in\Omega$, we can deduce that $\lm_1(\mathcal{P},\tilde N)<0$.

(ii) Now we prove uniqueness. Without loss of the generality we assume that $\Omega=(0,l)$. Let $(S_1,I_1,R_1)$ and $(S_2,I_2,R_2)$ be positive solutions of \eqref{3.1}. Since both $S_1+I_1+R_1$ and $S_2+I_2+R_2$ satisfy \eqref{3.3}, we have $\tilde N=S_1+I_1+R_1=S_2+I_2+R_2$ by the uniqueness of positive solutions of \eqref{3.3}.

Let $V_i=S_i+I_i$ (where $i=1, 2$), then $(V_i,I_i)$ satisfies
\bes\left\{\begin{array}{lll}\label{3.12}
-d\Delta V_i=a\tilde N-\gamma I_i-(\beta+b\tilde N)V_i,\;\;&x\in(0,l),\\[1mm]
-d\Delta I_i=k\mathcal{P}[I_i]V_i-(\gamma +\beta+b \tilde N+k\mathcal{P}[I_i])I_i,\;\;&x\in(0,l),\\[1mm]
V_i=I_i=0,\;\;&x=0,l.\\[1mm]
\end{array}\right.
\ees

 We first consider the case $\mathcal{P}[I_1]\ge \mathcal{P}[I_2]$.
Let $V=V_1-V_2$, $I=I_1-I_2$,
then $(V,I)$ satisfies
\bes\left\{\begin{array}{lll}\label{3.13}
-d\Delta V+(\beta+b\tilde N)V=-\gamma I,\;\;&x\in(0,l),\\[1mm]
-d\Delta I+k(I_2-V_2)\mathcal{P}[I]+(\gamma+\beta+b\tilde N+k\mathcal{P}[I_1])I=k\mathcal{P}[I_1]V,\;\;&x\in(0,l),\\[1mm]
V=I=0,\;\;&x=0,l
\end{array}\right.
\ees
We shall prove $V\equiv I\equiv 0$. Assume on the contrary that either $V\not\equiv 0$ or $I\not\equiv 0$. Then $V\not\equiv 0$ and $I\not\equiv 0$. In deed, $V\equiv 0$ implies $I\equiv 0$, and $I\equiv 0$ implies $V\equiv 0$ by \eqref{3.13} since $\gamma>0$ and $k\mathcal{P}[I_1]>0$.

For the open interval ${\cal I}\subset(0,l)$, we denote
 \bess
 (\mathscr{L}_1, {\cal I})u&=&-d\Delta u+(\beta+b \tilde N)u,\\[1mm]
(\mathscr{L}_2, {\cal I})u&=&-d\Delta u+k(I_2-V_2)\mathcal{P}[u]+(\gamma+\beta+b\tilde N+k\mathcal{P}[I_1])u,\\[1mm]
(\mathscr{L}_3, {\cal I})u&=&-d\Delta u-kV_2\mathcal{P}[u]+(\gamma+\beta+b\tilde N+k\mathcal{P}[I_2])u,\\[1mm]
{\cal D}(\mathscr{L}_i, {\cal I})&=&H^2({\cal I})\cap H_0^1({\cal I}),\;\; i=1, 2, 3.
\eess
Let $\lm_1(\mathscr{L}_i, {\cal I})$ be the principal eigenvalue of the following eigenvalue problem
\bess\left\{\begin{array}{ll}
	\mathscr{L}_i\phi=\lm\phi,\;\;&x\in {\cal I},\\[1mm]
	\phi=0,\;\;&x\in\partial{\cal I},
\end{array}\right.
\eess
$i=1,2, 3$. Then $\lm_1(\mathscr{L}_3, (0,\,l))=0$ as $I_2>0$ satisfies the second equation of \eqref{3.12}. Thanks to $I_2>0$ and $\mathcal{P}[I_1]\ge \mathcal{P}[I_2]$, it follows from Proposition \ref{p3.2} (i) that
\bess
\lm_1(\mathscr{L}_2, (0,\,l))>\lm_1(\mathscr{L}_3, (0,\,l))=0,
\eess
and from the monotonicity of $\lm_1(\mathscr{L}_1,(0,\,l))$ that
\bess
\lm_1(\mathscr{L}_1,(0,\,l))>\lm_1(\beta-a+b \tilde N)=0,\;\;\;\text{and}\;\; \lm_1(\mathscr{L}_1,{\cal I})>\lm_1(\mathscr{L}_1,(0,\,l))\;\;\,\text{when}\;\;{\cal I}\varsubsetneqq(0,\,l).\eess
By using Proposition \ref{p3.2} (i) and (ii) in sequence, it can be concluded that
\bess
\lm_1(\mathscr{L}_2, {\cal I})>\lm_1(\mathscr{L}_3, {\cal I})\ge\lm_1(\mathscr{L}_3, (0,\,l))=0.
\eess
Therefore,  when ${\cal I}\subset(0,\,l)$, the strong maximum principle holds for $(\mathscr{L}_1, {\cal I})$ by Proposition \ref{p3.1}, and holds for $(\mathscr{L}_2, {\cal I})$ by Proposition \ref{p3.3}.

We first show that $V$ must change signs. If $V\le 0$, then $I<0$ in $(0,l)$ by the strong maximum principle as $V\not\equiv 0$. Then by the first equation of \eqref{3.13}
the strong maximum principle shows that $V>0$ in $(0,l)$. This contradicts with the assumption that $V\le 0$. Similarly, $V\ge 0$ is also impossible. So, $V$ must change signs.

Take advantage of the above facts, similar to the proof of \cite[Lemmas 3.3 and 3.4]{LP}, we can prove the following claims (see also the proof of \cite[Theorem 3.5]{WangYao}).

{\bf Claim 1}: The set of zeros of $V$ is discrete in $[0,\,l]$.

{\bf Claim 2}: Let $[x_1,x_2]\subset[0,l]$ with $x_1<x_2$, and  $ V(x_1)=V(x_2)=0$. If $V>0$ in $(x_1, x_2)$ and $I(x_1)\ge 0$, then $ I(x_2)<0$; If $V<0$ in $(x_1, x_2)$ and $I(x_1)\le 0$, then $ I(x_2)>0$.

According to the above Claim 1 we see that the set of zeros of $V$ is discrete in $[0,\,l]$. Let $0=x_0<x_1<\cdots<x_n=l$ be such a set. Then $\lm_1(\mathscr{L}_1, (x_{j-1}, x_j))>0$, $\lm_1(\mathscr{L}_2, (x_{j-1}, x_j))>0$ for $j=1,\cdots, n$.
We shall follow the idea of \cite[Theorem 3.1]{LP} to derive a contradiction.

Without loss of generality we may think of that $V(x)>0$ in $(x_0,x_1)$. Then $I(x_1)<0$ by Claim 2. We assert that $V(x)<0$ in $(x_1, x_2)$. In fact, if  $V(x)>0$ in $(x_1, x_2)$, then there exists $\varepsilon>0$ such that $I(x)<0$ in $(x_1-\varepsilon, x_1+\varepsilon)$, and $ V(x_1\pm\varepsilon)>0$. Applying the strong maximum principle to the first equation of \eqref{3.13} in $(x_1-\varepsilon, x_1+\varepsilon)$, we can derive that $V(x)>0$ in $(x_1-\varepsilon, x_1+\varepsilon)$. This is a contradiction with the fact that $V(x_1)=0$.

We have known that $V(x)<0$ in $(x_1, x_2)$, $V(x_1)=V(x_2)=0$ and $ I(x_1)<0$, it follows that $I(x_2)>0$ by Claim 2. By repeating this process, we can ultimately deduce that either $I(l)>0$ or $I(l)<0$. This contradiction implies that $V\equiv I\equiv 0$. As a result, if $N=1$, then the positive solution of \eqref{3.1} is unique when it exists.

Now we consider the case $R_1\ge R_2$ in $\Omega$. In this case, $V_1\le V_2$ as $V_1+R_1=V_2+R_2$, which implies $I_1<V_1\le V_2$.
Writing the second equation of \eqref{3.13} as
\bess
-d\Delta I+k(I_1-V_2)\mathcal{P}[I]+(\gamma+\beta+b\tilde N+k\mathcal{P}[I_2])I=k\mathcal{P}[I_1]V,\eess
and define the operator $(\mathscr{L}_2, {\cal I})$ as
\bess
(\mathscr{L}_2, {\cal I})u=-d\Delta u+k(I_1-V_2)\mathcal{P}[u]+(\gamma+\beta+b\tilde N+k\mathcal{P}[I_2])u.
\eess
We have known $I_1<V_2$. Similar to the above, the strong maximum principle holds for $(\mathscr{L}_2, {\cal I})$ when ${\cal I}\subset(0,\,l)$. The remaining proof is the same as above and we omit the details.
\end{proof}

\section{The dynamical properties of the problem \eqref{1.5}}
\setcounter{equation}{0} {\setlength\arraycolsep{2pt}

Throughout this section we always assume that hypotheses {\bf(P2)} and {\bf(F)} hold.

\subsection{The well-posedness of the solution of the problem \eqref{1.5}}

\begin{theo}\lbl{t4.1} The problem \eqref{1.5} has an unique global solution $(S,I,R;g,h)$ and for any $0<\alpha<1$, there exists a constant $C$ such that
\bess
&0<S\le C\;\;\text{in}\;\;\R_+\times\R,\;\;0<I,R\le C\;\;\text{in}\;\;D_\infty,\;\;0<-g', h'\le
C\;\;\text{in}\;\;\R_+,&\\
&S+I+R\le \max\kk\{\|S_0\|_{L^{\yy}(\mathbb{R})}+\|I_0\|_{C([-h_0,\,h_0])},
(a-\beta)/b\rr\}=:A\;\;in\;\,[0,\infty)\times\mathbb{R},&
  \eess
and
 \bes
\|I(t,\cdot),R(t,\cdot)\|_{C^1([g(t),\,h(t)])}\leq C\ \ \forall\,t\ge1,\quad \|g', h'\|_{C^{{\frac{\alpha}{2}}}([1,\yy))}\leq C,\lbl{4.1}
 \ees
Moreover, for any $0<T<\infty$ and $0<\alpha<1$,we also have
 \[(S,I,R;g,h)\in C_{\rm loc}^{(1+\alpha)/2,\,1+\alpha}([0,T]\times\R)\times [C^{(1+\alpha)/2,\,1+\alpha}_{\rm loc}(\bar D_T)]^2\times [C^{1+\alpha/2}_{\rm loc}([0,T])]^2,\]
where $D_T=\{(t,x):\, 0\le t\le T, \, g(t)<x<h(t)\}$.
\end{theo}

The proof of Theorem 4.1 is akin to that of \cite[Theorem 2.1]{ChenW}, and we omit the details. As we can see from Theorem \ref{t4.1}, the length $h(t)-g(t)$ of habitat of the infected population $I$ is increasing in $t>0$, which obviously implies that either $\dd\lim_{t\to\yy}(h(t)-g(t))<\yy$ or $\dd\lim_{t\to\yy}(h(t)-g(t))=\yy$. We say that spreading occurs if $\dd\lim_{t\to\yy}(h(t)-g(t))=\yy$, and vanishing occurs if $\dd\lim_{t\to\yy}(h(t)-g(t))<\yy$. For the sake of convenience, we define $\dd g_{\yy}=\lim_{t\to\yy}g(t)$ and $\dd h_{\yy}=\lim_{t\to\yy}h(t)$.

\subsection{Longtime behaviors and criteria for spreading and vanishing}

We first discuss the longtime behaviors of solution of \eqref{1.5}.

\begin{theo}\lbl{t4.2} Let $(S, I, R; g, h)$ be the solution of \eqref{1.5}. If $h_\infty-g_\infty<\infty$, then
 \bes
\lim_{t\to\infty} \max_{[g(t),h(t)]}I(t,\cdot)=0, \;\;\lim_{t\to\infty} \max_{[g(t),h(t)]}R(t,\cdot)=0,\lbl{4.2}\ees
and
 \bes
 \lim_{t\to\infty} S(t,\cdot)=(a-\beta)/b\;\;\text{locally uniformly in } \; \mathbb{R}.\lbl{4.3}\ees
\end{theo}

\begin{proof} \textbf{Step 1.} Recalling $N=S+I+R$, we have
  \bes\left\{\begin{array}{ll}\lbl{4.4}
 N_t-d N_{xx}=aN-\beta N-bN^2,\;\;t>0,\; x\in\mathbb{R}\setminus\{h(t),g(t)\},\\[1mm]
N_x(t,g(t)-0)\le N_x(t,g(t)+0),\; \; N_x(t,h(t)-0)\le N_x(t,h(t)+0),\;\;t>0,\\ [1mm]
 N(0,x)=S_0(x)+I_0(x), \;\; x\in\mathbb{R}.
 \end{array}\right.\ees
Thus $N$ can be considered as a weak lower solution of the problem
 \bess\left\{\begin{array}{ll}
 \bar N_t-d\bar N_{xx}=a\bar N-\beta \bar N-b\bar N^2,\;\;&t>0,\; x\in\mathbb{R},\\[1mm]
 \bar N(0,x)=S_0(x)+I_0(x), \;\; &x\in\mathbb{R}.
 \end{array}\right.\eess
As a result, $N(t,x)\le \bar N(t,x)$.

Let $q(t)$ be the unique solution of the initial value problem
 \bess\left\{\begin{array}{ll}
	 q'(t)=(a-\beta-b q)q,\;\;t>0,\\[1mm]
	 q(0)=\|S_0\|_{L^\infty(\mathbb{R})}+\|I_0\|_{C([-h_0,h_0])
}>0.
 \end{array}\right.
 \eess
Then $\lim\limits_{t\to\infty}q(t)=(a-\beta)/b$. The comparison principle yields $\bar N(t,x)\le q(t)$, further we have $N(t,x)\le q(t)$ for $t\ge 0$ and $x\in\mathbb{R}$.
For the given $\varepsilon>0$, there exists $T\gg 1$ such that
\bes
N(t,x)\le(a-\beta)/b+\varepsilon,\;\;\forall\;t\ge T,\; x\in\mathbb{R}.\lbl{4.5}
\ees
Thus $I$ satisfies
 \[I_t-d I_{xx}\ge -(\gamma+\beta+bN)I\ge -(a+\gamma+b\varepsilon)I,\;\;\forall\;t>T,\;x\in(g(t),h(t)),\]
and the first estimate in \eqref{4.1} holds. Using \cite[Lemma 8.7]{Wangpara} or \cite[Lemma 4.1]{wzjde18} we have
 $$\lim_{t\to\infty} \max_{[g(t),h(t)]}I(t,x)=0, .$$
Thanks to the equation of $R$, it is easy to show the second limit in \eqref{4.2}.

\textbf{Step 2.} We have obtained in Step 1 that
\bes
\dd\limsup _{t\to\infty } S(t,x)\le\limsup_{t\to\infty }N(t,x)\le(a-\beta)/b\;\;\text{ uniformly in } \; \mathbb{R}.\lbl{4.6}
\ees
Since $I(t,x)=R(t,x)=0$ for $x\notin(g(t),h(t))$. Using the fact \eqref{4.2} we see that for the small $\sigma>0$, there exists $T_\sigma\gg1$, such that
\[-\sigma<I(t,x),R(t,x)<\sigma,\;\;\forall\;t>T_\sigma,\;x\in\mathbb{R}.\]
For any given $\varepsilon>0$ and $L>0$, let $l_\varepsilon$ be determined by \cite[Proposition B.1]{WZjdde17}. Then S satisfies
\bess\left\{\begin{array}{ll}
S_t-d S_{xx} \ge (a-\beta-bS-2b\sigma-k\sigma)S,\;\;
&t>T_\sigma,\;-l_\varepsilon<x<l_\varepsilon,\\[1mm]
S(t,\pm l_\varepsilon)>0,\;\;&t\ge T_\sigma,\\[1mm]
S(T_\sigma,x)>0,\;\;
&-l_\varepsilon\le x\le l_\varepsilon.
\end{array}\right.
\eess
By \cite[Proposition B.1]{WZjdde17}
\bess
\liminf_{t\to\infty}S(t,x)\ge\frac{a-\beta-(2b+k)\sigma}{b}-
\varepsilon\;\;\;\text{uniformly in}\;\; [L,L].
\eess
By the arbitrariness of $\varepsilon$, $\sigma$ and $L$, we have
 \bess
\liminf_{t\to\infty}S(t,x)\ge(a-\beta)/b\;\;\;\text{locally uniformly in} \;\;\mathbb{R}.
 \eess
This combined with \eqref{4.6} yields \qq{4.3}.
The proof is complete. \end{proof}

In order to make sure that the spreading and vanishing of the disease, we study the eigenvalue problem:
\begin{equation}\lbl{4.7}
\begin{cases}
-d \phi''-c_1{\dd\int_{\mathcal I}P(x, y)\phi(y){\rm d}y}+c_2\phi=\lambda\phi, &x\in\mathcal{I},\\
\phi(x)=0, &x\in\partial\mathcal{I},
\end{cases}
\end{equation}
where $d, c_1$ and $c_2$ are positive constants, and $\mathcal{I}\subset \mathbb{R}$ is an open interval. By Krein-Rutman theorem, we effortlessly prove that \eqref{4.7} has a unique principal eigenvalue, defined by
$\lambda_1(c_1,c_2, \mathcal{I})$. By the variational method, $\lambda_1(c_1,c_2, \mathcal{I})$ can be expressed by
 \begin{equation*}
   \lambda_1(c_1,c_2, \mathcal{I})=\inf_{\substack{\phi\in h_0^1(\mathcal{I})\\
\|\phi\|_2=1}}\left\{d\int_\mathcal{I}(\phi')^2+c_2-c_1\iint_{\mathcal{I}
 \times\mathcal{I}}P(x, y)\phi(y)\phi(x)\,{\rm d}y{\rm d}x\right\}.
   \end{equation*}
Some properties of $\lambda_1(c_1,c_2, \mathcal{I})$ were given in \cite[Proposition B.1]{HWzamp19}. Set
    \bess
  \mathcal{R}_{02}(c_1,c_2, \mathcal{I})=\sup_{\substack{\phi\in h_0^1(\mathcal{I})\\		 \|\phi\|_2=1}}\left\{\frac{c_1\iint_{\mathcal{I}\times\mathcal{I}}P(x,y)\phi(y)\phi(x)\,{\rm d}y{\rm d}x}{d\int_\mathcal{I}(\phi')^2+c_2} \right\}.
   \eess
Then $\lambda_1( c,\mathcal{I})>\,(=,\,<)\,0$ if and only if $\mathcal{R}_{02}(c_1,c_2, \mathcal{I})<\,(=,\,>)\,1$.

\begin{theo}\lbl{t4.3}\, Let $(S, I, R; g, h)$ be the solution of \eqref{1.5}. If $\mathcal{R}_{02}(\frac{k(a-\beta)}{b},a+\gamma,\,(-h_0,h_0))\ge 1$, then
\bes
h_\infty=-g_\infty =\infty \;\;and \;\;\dd\lim_{t\to\infty}\|I({\cdot},t)\|_{ C([g(t),h(t)])}>0.\nonumber
\ees
\end{theo}

\begin{proof} To save spaces, we denote $g_\infty^\varepsilon=g_\infty+\varepsilon$ and $h_\infty^\varepsilon=h_\infty-\varepsilon$.\vskip 2pt

{\it Case 1}: $\mathcal{R}_{02}(\frac{k(a-\beta)}{b},a+\gamma,(-h_0,h_0))>1$. Thus we have   $\lambda_1(\frac{k(a-\beta)}{b},a+\gamma,(-h_0,h_0))<0$.

\textbf{Step 1.}  We prove $h_\infty-g_\infty =\infty$. Assuming on the contrary that $h_\infty-g_\infty <\infty$. Then we have that \eqref{4.2} and \eqref{4.3} hold by Theorem \ref{t4.2}. Moreover, combining with \eqref{4.5}, there exists small $\varepsilon>0$ such that  $\lambda_1(k(\frac{a-\beta}{b}-\varepsilon),a+\gamma+\varepsilon,
(g_\infty^\varepsilon,h_\infty^\varepsilon))<0$ by the properties of $\lambda_1(c_1,c_2,\mathcal{I})$. For this $\varepsilon$, there exists $T^*>0$ such that $g(t)<g_\infty^\varepsilon$, $h(t)>h_\infty^\varepsilon$, and $S(t,x)\ge \frac{a-\beta}{b}-\varepsilon$ by \qq{4.3}, and $N(t,x)\le \frac{a-\beta}{b}+\varepsilon$ by \qq{4.5} for all $t\ge T^*$ and $x\in [g_\infty^\varepsilon,h_\infty^\varepsilon]$. Thus $I$ satisfies
  \bes\left\{\begin{array}{ll}
  I_t-d I_{xx}\dd\ge k\kk(\frac{a-\beta}{b}-{\varepsilon}\rr)\int_{g_\infty^\varepsilon}^{h_\infty^\varepsilon}P(x,y)I(t,y)\,{\rm d}x-(\gamma
  +a+\varepsilon)I,&t\geq T^*,\; x\in(g_\infty^\varepsilon,h_\infty^\varepsilon),\\[3mm]
  I(t,g_\infty^\varepsilon)> 0,\; I(t,h_\infty^\varepsilon)> 0, &t\ge T^*.
  \end{array}\right.\qquad
  \lbl{4.8}\ees
Set
 \[\underline{I}(t,x)=\delta\phi_\varepsilon(x) \; \; \text{for all}\ t\ge T^*,\; x\in [g_\infty^\varepsilon,h_\infty^\varepsilon],\]
where $0<\phi_\varepsilon(x)\le 1$ is the eigenfunction corresponding to $\lambda_1(k(\frac{a-\beta}{b}-\varepsilon),a+\gamma+\varepsilon,
(g_\infty^\varepsilon,h_\infty^\varepsilon))$, $\delta$ is a
enough small positive constant such that
$\delta\phi_\varepsilon(x)<I(T^*,x)$ in $ [g_\infty^\varepsilon,h_\infty^\varepsilon]$. Obviously, $\delta\phi_\varepsilon(g_\infty^\varepsilon)=0<I(t,g_\infty^\varepsilon)$ and $\delta\phi_\varepsilon(h_\infty^\varepsilon)=0<I(t,h_\infty^\varepsilon)$.  Directly calculate yields
  \begin{align*}
 &\underline{I}_t-d\underline{I}_{xx}-k\dd\kk((a-\beta)/b-\varepsilon\rr)\mathcal{P} [\underline{I}]+(a+\gamma+\varepsilon)\underline{I}\\[1mm]
 &=\delta\lambda_1\left(k\kk((a-\beta)/b-\varepsilon\rr),a+\gamma+\varepsilon,
(g_\infty^\varepsilon,h_\infty^\varepsilon)\right)\phi_\varepsilon(x)\\[1mm]
&<0
\end{align*}
for $t\ge T^*$ and $x\in(g_\infty^\varepsilon,h_\infty^\varepsilon)$. By the comparison principle, we get
 \bess
  \liminf_{t\to\infty}I(t,\cdot)\ge\liminf_{t\to\infty}\underline{I}(t,\cdot)
  =\delta\phi_\varepsilon(\cdot)>0\;\;\mbox{in }\; (g_\infty^\varepsilon,h_\infty^\varepsilon).
  \eess
This contradicts with the first limit of \eqref{4.2}. Therefore $h_\infty-g_\infty=\infty$.

\textbf{Step 2.}  We prove $h_\infty=-g_\infty =\infty$. Assuming on the contradiction that $h_\infty<\infty$. Thanks to $h_\infty-g_\infty =\infty$, we have $g_\infty=-\infty$. For any given $-\infty<l<h_0$ and $0<\varepsilon\ll1$, there exists $T\gg 1$ such that $g(t)<l$ and $N(t,x)\le \frac{a-\beta}{b}+\varepsilon$ for all $t\ge T$ and $x\in(l,h(t))$. Thus $(I, h(t))$ satisfies
 \bess\left\{\begin{array}{ll}
 I_t-d I_{xx}\ge -(a+\gamma+\varepsilon)I,\;\;&t>T,\;x\in(l,h(t)),\\[1mm]
 I(t,l)\ge 0,\,\,I(t, h(t))=0,&t>T,\\[1mm]
 h'(t)=-\mu I_x(t,h(t)),&t>T,\\[1mm]
 I(T,x)>0,&x\in(l, h(T)).
 \end{array}\right.
 \eess
Using \cite[Lemma 8.7]{Wangpara} or \cite[Lemma 4.1]{wzjde18} we get
$\dd\lim_{t\to\infty}\max_{l\le x\le h(t)}I(t,x)=0$.
This means that
 \bes
 \dd\lim_{t\to\infty}I(t,\cdot)=0\;\;\mbox{ locally uniformly in }\; (-\infty, h_\infty).
 \lbl{4.9}\ees
Same as the proof of Theorem \ref{t4.2} we can prove that $$\dd\lim_{t\to\infty}R(t,\cdot)=0,\,\,\, \lim_{t\to\infty} S(t,\cdot)=(a-\beta)/b\;\;\mbox{locally uniformly in }\; (-\infty, h_\infty).$$

For any given $l_1<l_2$ satisfying $(l_1,l_2)\supset[-h_0, h_0]$ and  $[l_1,l_2]\subset(-\infty,h_\infty)$, we can get   $\lambda_1(k(\frac{a-\beta}{b}-\varepsilon),a+\gamma+\varepsilon,
(l_1,l_2))<0$ when $0<\varepsilon\ll 1$. Moreover, there exists $T^*>0$ such that  $g(t)<l_1$, $h(t)>l_2$,  $S(t,x)\ge \frac{a-\beta}{b}-\varepsilon$ and $N(t,x)\le \frac{a-\beta}{b}+\varepsilon$ for all $t\ge T^*$ and $x\in [l_1,l_2]$. Therefore $I$ satisfies \eqref{4.8} with $g_\infty^\varepsilon$ and $h_\infty^\varepsilon$ replacing by $l_1$ and $l_2$, respectively. Set
 \[\underline{I}(t,x)=\phi(x) \; \ \text{for}\ t\ge T^*,\; x\in [l_1,l_2],\]
where $0 < \phi(x) \le 1$ is the eigenfunction corresponding to $\lambda_1((k(\frac{a-\beta}{b}-\varepsilon),a+\gamma+\varepsilon,(l_1,l_2))$. Same as  above, we can prove that
\bess
  \liminf_{t\to\infty}I(t,\cdot)>0\;\;\mbox{in }\; (l_1,l_2).
  \eess
This contradicts with \eqref{4.9}. Similarly, we can also prove $g_\infty=-\infty$.

The above argumentation also means $\dd\lim_{t\to\infty}\|I({\cdot},t)\|_{ C([g(t),h(t)])}>0$.

{\it Case 2}: $\mathcal{R}_{02}((k(\frac{a-\beta}{b}), a+\gamma,\,(-h_0,h_0))=1$. In this case, $\lambda_1((k(\frac{a-\beta}{b}),a+\gamma,(-h_0,h_0))=0$. Therefore,   $\lambda_1((k(\frac{a-\beta}{b}),a+\gamma,(g(t_0),h(t_0))<0$  for any $t_0>0$ by the monotonicity of $\lambda_1(c_1,c_2,\mathcal{I})$ in $\mathcal{I}$. Replacing the initial time $0$ by $t_0$, similar to Case 1, we can get $h_\infty-g_\infty=\infty$.
\end{proof}

It follows from \cite[Propositon B.1]{HWzamp19} that there exists $l^*$ such that
  \bess\left\{\begin{array}{ll}
 \lambda_1(\frac{k(a-\beta)}{b},a+\gamma,(l_1,l_2))=0\;\;&{\text{if}}\ \  l_2-l_1=l^*,\\[1mm]
  \lambda_1(\frac{k(a-\beta)}{b},a+\gamma,(l_1,l_2))>0\;\;&{\text{if}}\ \   l_2-l_1<l^*,\\[1mm]
  \lambda_1(\frac{k(a-\beta)}{b},a+\gamma,(l_1,l_2))<0\;\;&{\text{if}}\ \   l_2-l_1>l^*.
\end{array}\right.
\eess
So we can obtain the following property easily.

\begin{col}\lbl{c4.1}\, If $h_{\infty}-g_{\infty}<\infty$, then $h_{\infty}-g_{\infty}\le l^*$.
Thus $2h_0>l^*$ implies $h_{\infty}-g_{\infty}=\infty$.
\end{col}

\begin{proof} Indirect argument we assume $h_{\infty}-g_{\infty}>l^*$, then we have   $\lambda_1(\frac{k(a-\beta)}{b},a+\gamma,(g_{\infty},h_{\infty}))<0$. Similar to the proof of Theorem \ref{t4.3} we can prove  $\dd\liminf_{t\to\infty}I(t,x)>0$,\;$x\in[g(t),h(t)]$, which implies $\dd\lim_{t\to\infty}\|I({\cdot},t)\|_{ C([g(t),h(t)])}>0$. This contradicts with the first limit of \eqref{4.2}. Therefore $h_{\infty}-g_{\infty}=\infty$ by Theorem \ref{t4.2}. This is a contradiction.
\end{proof}

\begin{theo}\lbl{t4.5} Let $\|S_0\|_{L^\infty(\mathcal{R})}+ \|I_0\|_{C([-h_0,h_0])
}\le(a-\beta)/b$ and $\mathcal{R}_{02}(\frac{k(a-\beta)}{b},\gamma+\beta,\,(-h_0,h_0))<1$. Then there exist $\mu^*\geq\mu_*>0$, such that $h_\infty-g_\infty=\infty$ when $\mu>\mu^*$, and $h_\infty-g_\infty\leq l^*$ when $\mu\leq\mu_*$ or $\mu=\mu^*$.
\end{theo}

\begin{proof} Making use of \eqref{4.4} and the condition $\|S_0\|_{L^\infty(\mathbb{R})}+ \|I_0\|_{C([-h_0,h_0])}\le\frac{a-\beta}{b}$, we can easily obtain that $0<N(t,x)\le \frac{a-\beta}{b}$ for all $t\ge 0$ and $x\in\mathbb{R}$.
Thus $I$ satisfies
  \bess
  I_t-d I_{xx}\le\frac{k(a-\beta)}{b}\mathcal{P}[I]-(\gamma+\beta)I,&\;\;\forall\;t> 0,\; x\in(g(t), h(t)).
  \eess

\textbf{Step 1.} We first show that $h_\infty-g_\infty<\infty$ when $\mu$ is small. The proof process is similar with that in \cite[Theorem 3.3]{HWzamp19} for slightly revises. For the convenience of readers, let's give the details. By  $\mathcal{R}_{02}(\frac{k(a-\beta)}{b},\gamma+\beta,\,(-h_0,h_0))<1$, we have $\lambda_1(\frac{k(a-\beta)}{b},\gamma+\beta,(-h_0,h_0))>0$. There exists a $\varepsilon>0$ such that $\lambda_\varepsilon:=\lambda_1(k(\frac{a-\beta}{b}+\varepsilon),\gamma+\beta,(-h_0,h_0))>0$. Let $\phi(x)$ be the corresponding positive eigenfunction to $\lambda_\ep$ with $\|\phi\|_\infty=1$. Then $\phi(x)$ satisfies
 \bes\lbl{4.10}
\begin{cases}
-d\phi''-k(\frac{a-\beta}{b}+\varepsilon)\mathcal{P}[\phi]+(\gamma+\beta)\phi=\lambda_\ep \phi,\ \ x\in (-h_0,h_0),\\
\phi(\pm h_0)=0.
\end{cases}
 \ees
It follows that $\phi\in C^{2+\alpha}([-h_0, h_0])$ and
 \bess
\begin{cases}
-d\phi''+(\gamma+\beta)\phi=\lambda_\ep\phi+k(\frac{a-\beta}{b}+\varepsilon)
\mathcal{P}[\phi]>0,\ \ x\in (-h_0,h_0),\\
\phi(\pm h_0)=0.
\end{cases}
 \eess
Therefore, $\phi'(-h_0)>0$ and $\phi'(h_0)<0$ by the Hopf boundary lemma.

We define
  $$\begin{cases}
\sigma(t)=1+2\delta-\delta{\rm e}^{-\delta t},\quad s(t)=h_0 \sigma(t),&t>0,\\[1mm]
\bar I(t,x)=A {\rm e}^{-\delta t}\phi(\frac{x}{\sigma(t)}),&t>0\ x\in[-s(t),s(t)],
\end{cases}$$
where $\delta, A>0$ are two constants to be determined. It is clear that $-s(0)\le -h_0$ and $s(0)\ge h_0$. We shall check that $(\bar I(t,x), s(t))$ satisfies
\begin{align}
&\bar{I}_t-d\bar I_{xx}\ge\frac{k(a-\beta)}{b}\mathcal{P}[\bar I]-(\gamma+\beta)\bar{I}\quad \text{for}\; t>0,\ x\in(-s(t),s(t)),\lbl{4.11}\\
&\bar{I}(t,\pm s(t))=0,\ \ -s'(t)\le -\mu \bar{I}_x(t,- s(t)),\ \ s'(t)\ge -\mu \bar{I}_x(t,  s(t)) \quad \text{for}\; t>0,\lbl{4.12}\\
&\bar{I}(0,x)\ge I(0,x)\quad \text{for}\; x\in[-h_0,h_0],\lbl{4.13}
\end{align}
when $\mu$ is small.

First, it is easily to see that \eqref{4.13} holds for sufficient large $A$.

Next, we check \eqref{4.12}. Direct computations show that
\begin{align*}
| \bar I_x(t,\pm s(t))|= A{\rm e}^{-\delta t}|\phi'(\pm h_0)|/\sigma(t)\le  A{\rm e}^{-\delta t}|\phi'(\pm h_0)|/(1+\delta) \ \ \text{and}\ \ s'(t)=h_0\delta^2{\rm e}^{-\delta t}.
\end{align*}
This implies that \eqref{4.12} holds if
 $$\mu\le \frac{h_0\delta^2(1+\delta)}{A|\phi'(\pm h_0)|}:=\mu_0.$$

At last, we clarify \eqref{4.11}. Make use of \eqref{4.10} and $\sigma(t)>1$, by the carefully calculations we have
 \begin{align*}
 &\ \bar{I}_t-d\bar I_{xx}-\frac{k(a-\beta)}{b}\mathcal{P}[\bar I]+(\gamma+\beta)\bar{I}\\
 =&\ A {\rm e}^{-\delta t}\left(-\delta\phi(z)-\phi'(z)\frac{z\sigma'(t)}{\sigma(t)}-
d \frac{\phi''(z)}{\sigma^2(t)}-\frac{k(a-\beta)}{b}\int_\mathbb{R} P(\sigma(t) z,y)\phi\left(\frac{y}{\sigma(t)}\right){\rm d}y+(\gamma+\beta)\phi(z)\right)\\[1mm]
\ge&\ A {\rm e}^{-\delta t}\left(-\delta\phi(z)-\phi'(z)\frac{z\sigma'(t)}{\sigma(t)}-
 d\frac{\phi''(z)}{\sigma^2(t)}-\frac{k(a-\beta)}{b}\int_\mathbb{R} P(\sigma(t) z,y)\phi\left(\frac{y}{\sigma(t)}\right){\rm d}y+\frac {\gamma+\beta}{\sigma^2(t)}\phi(z)\right)\\[1mm]
 =&\ A {\rm e}^{-\delta t}\left(\frac{\lambda_\varepsilon}{\sigma^2(t)}\phi(z)-\delta\phi(z)-\phi'(z)
\frac{z\sigma'(t)}{\sigma(t)}+F(t,z)\right),
\end{align*}
where $z=x/\sigma(t)\in(-h_0,h_0)$ and
 \[F(t,z)=\frac{k(\frac{a-\beta}{b}+\varepsilon)}{\sigma^2(t)}\mathcal{P}[\phi](z)-\frac{k(a-\beta)}{b}\int_\mathbb{R} P(\sigma(t) z,y)\phi\left(\frac{y}{\sigma(t)}\right){\rm d}y.\]
We now estimate $F(t,z)$. Direct calculation yields
 \begin{align*}
F(t,z)\ge&\frac{\varepsilon}{\sigma^2(t)}k\mathcal{P}[\phi](z)
-\frac{k(a-\beta)}{b}\left|\frac 1{\sigma^2(t)}\mathcal{P}[\phi](z)-\int_\mathbb{R} P(\sigma(t)z,y)\phi\left(\frac{y}{\sigma(t)}\right){\rm d}y\right|\\[1mm]
=&\frac{\varepsilon}{\sigma^2(t)}k\mathcal{P}[\phi](z)
-\frac{k(a-\beta)}{b}\left|\frac 1{\sigma^2(t)}\mathcal{P}[\phi](z)-\sigma(t)\int_\mathbb{R} P\left(\sigma(t)z,\sigma(t)y\right)\phi(y){\rm d}y\right|\\[1mm]
\ge& \frac{\varepsilon}{\sigma^2(t)}k\mathcal{P}[\phi](z)-\frac{k(a-\beta)}{b}\sigma(t)\!\int_{-h_0}^{h_0}\!\! \left|P(z,y)-P(\sigma(t)z,\sigma(t)y)\right|{\rm d}y-\frac{6\delta k(a-\beta)}{b}.
 \end{align*}
In view of the hypothesis {\bf(P2)} and the facts that $\phi\in C^{2+\alpha}([-h_0, h_0])$ and $\phi(y)>0$ for all $y\in(-h_0, h_0)$, it is easy to see that the function
 \bess
 \mathcal{P}[\phi](z)=\int_{-\infty}^\infty P(z,y)\phi(y){\rm d}y=\int_{-h_0}^{h_0} P(z,y)\phi(y){\rm d}y
 \eess
is continuous and positive in $[-h_0,h_0]$. Therefore,
  \[m=\min_{z\in[-h_0,h_0]}\frac{\varepsilon}{4}k\mathcal{P}[\phi](z)>0.\]
As $P$ is Lipschitz continuous in the rectangle $[-h_0(1+2\delta), h_0(1+2\delta)]^2$, there exists a constant $\delta^*\in(0,1/2)$ such that for any $0<\delta\le\delta^*$,
 \[\frac{k(a-\beta)}{b}\sigma(t)\int_{-h_0}^{h_0} \left|P(z,y)-P(\sigma(t)z,\sigma(t)y)
 \right|{\rm d}y\le \frac{m}{2}.\]
Notice that $\sigma(t)<2$ when $0<\delta<1/2$. It follows that $F(t,z)\ge 0$ in $[0,\infty)\times[-h_0,h_0]$ when $0<\delta\le \min\{\delta^*,  \frac{m}{12k(a-\beta)/b}\}$. For such $\delta$, we have
 \[ \bar{I}_t-d\bar I_{xx}-\frac{k(a-\beta)}{b}\mathcal{P}[\bar I]+(\gamma+\beta)\bar{I}\ge\ A {\rm e}^{-\delta t}\left(\frac{\lambda_\varepsilon}{\sigma^2(t)}\phi(z)-\delta\phi(z)-\phi'(z)
 \frac{z\sigma'(t)}{\sigma(t)}\right).\]
Since $\phi'(-h_0)>0$ and $\phi'(h_0)<0$, we can find a $\tau>0$ such that
 \[-\phi'\left(z\right)\frac{z\sigma'(t)}{\sigma(t)}\ge 0\quad \text{for}\ t>0,\ z\in [-h_0,-h_0+\tau]\cup[h_0-\tau,h_0].\]
Hence, for $ t>0$ and $z\in [-h_0,-h_0+\tau]\cup[h_0-\tau,h_0]$, we can get
 \begin{align*}
   \bar{I}_t-d\bar I_{xx}-\frac{k(a-\beta)}{b}\mathcal{P}[\bar I]+(\gamma+\beta) \bar{I}\ge 0
 \end{align*}
provided that $0<\delta\ll 1$. On the other hand, for $t>0$ and $z\in[-h_0+\tau,h_0-\tau]$, we can get
 \bess
\bar{I}_t-d\bar I_{xx}-\frac{k(a-\beta)}{b}\mathcal{P}[\bar I]+(\gamma+\beta) \bar{I}
&\ge&\ A {\rm e}^{-\delta t}\kk[\kk(\frac{\lambda_\varepsilon}{\sigma(t)^2}-\delta\rr)
\min_{[-h_0+\tau,h_0-\tau]}\phi(z)-\|\phi'\|_\infty h_0\delta^2{\rm e}^{-\delta t}\rr]\\
&\ge& 0
 \eess
provided that $0<\delta\ll 1$. Thus \eqref{4.11} holds for the enough small $\delta$.

Now, we can use the comparison principle to gain
 \[h(t)-g(t)\le2s(t)\to 2h_0(1+2\delta)<\infty\;\; \text{as}\; t\to\infty.\]

\textbf{Step 2.} Since $I$ satisfies $I_t-d I_{xx}\ge -(a+\gamma)I$, we can obtain  $h_\infty-g_\infty>l^*$ when $\mu>\mu^*$ for some $\mu^*>0$ (\cite[Lemma 8.7]{Wangpara}, or \cite[Lemma 4.3]{wzdcds18}). Thus $h_\infty-g_\infty=\infty$ when $\mu>\mu^*$ by Corollary \ref{c4.1}.

\textbf{Step 3.} Define
\bess
\sum{_*}=\left\{\nu: \nu\ge \mu_0\; such \;that\;h_\infty-g_\infty\le l^*\;for \;all \;0<\mu\le \nu\right\}
\eess
where $\mu_0$ is given by step 1. Then $\mu_*:=sup\sum{_*}\le \mu^*$ and $(0,\mu_*)\subset\sum{_*}$.
We will prove that $\mu_*\in \sum{_*}$. Assuming on the contrary that  $h_{\mu_*,\infty}-g_{\mu_*,\infty}=\infty$. There exists $T>0$ such $h_{\mu_*}(T)-g_{\mu_*}(T)>l^*$. By the continuous dependence of $(S_\mu,I_\mu,R_\mu;g_\mu,h_\mu)$ on $\mu$, there \;is\; $\varepsilon>0$ such that $h_{\mu}(T)-g_{\mu}(T)>l^*$ for $\mu\in(\mu_*-\varepsilon,\mu_*+\varepsilon)$. It follows that for all such $\mu$,
\[\dd\lim_{t\to\infty}(h_{\mu}(t)-g_{\mu}(t))\ge h_{\mu}(T)-g_{\mu}(T)>l^*\]
We further have that $(\mu_*-\varepsilon,\mu_*+\varepsilon)\cap\sum{_*}=\emptyset$, and $sup\sum{_*}\le \mu_*-\varepsilon$. This contradicts the definition of $\mu_*$, so $\mu_*\in \sum{_*}$. Thus $h_\infty-g_\infty\leq l^*$ when $\mu\leq\mu_*$ or $\mu=\mu^*$.\end{proof}

\end{document}